\documentclass[12pt]{article}

\usepackage{graphicx}

\usepackage{epstopdf}

\usepackage{amsmath}
\usepackage{amsfonts}
\usepackage{amssymb}
\usepackage{amsfonts}
\usepackage{amsthm}
\usepackage{verbatim}
\usepackage{array}
\usepackage{xfrac}
\usepackage{mathtools}
\usepackage{hyperref}
\usepackage{xcolor}
\usepackage{mathrsfs}
\usepackage{mathtools}
\usepackage[scr=boondox,  
            cal=esstix]   
           {mathalpha}
\usepackage{xfrac}  
\usepackage{dsfont}
\usepackage{cite}

\numberwithin{equation}{section}

\theoremstyle{plain}
\newtheorem{theorem}{Theorem}[section]
\newtheorem{lemma}[theorem]{Lemma}

\newtheorem{cor}[theorem]{Corollary}
\theoremstyle{definition}

\newtheorem{remark}[theorem]{Remark}

\newcommand{\bigslant}[2]{{\raisebox{.1em}{$#1$}\left/\raisebox{-.1em}{$#2$}\right.}}

\newcommand{\torus}{{\mathbb T}}
\newcommand{\per}{\mathcal T}
\newcommand{\R}{{\mathbb R}}
\newcommand{\N}{{\mathbb N}}
\newcommand{\Z}{{\mathbb Z}}
\newcommand{\dd}{\mathrm{d}}
\newcommand{\loc}{\mathrm{loc}}

\DeclareMathOperator{\supp}{supp}
\DeclareMathOperator*{\esssup}{ess\,sup}

\newcommand{\np}[1]{(#1)}
\newcommand{\nb}[1]{[#1]}
\newcommand{\bp}[1]{\big(#1\big)}
\newcommand{\bb}[1]{\big[#1\big]}
\newcommand{\Bp}[1]{\bigg(#1\bigg)}

\newcommand{\norm}[1]{\lVert #1 \rVert}
\newcommand{\norml}[1]{\big\lVert #1 \big\rVert}
\newcommand{\snorm}[1]{\lvert #1 \rvert}
\newcommand{\snorml}[1]{\big\lvert #1 \big\rvert}
\newcommand{\snormL}[1]{\bigg\lvert #1 \bigg\rvert}

\newcommand{\set}[1]{\{ #1 \}}
\newcommand{\setl}[1]{\big\{ #1 \big \}}

\newcommand{\setc}[2]{\{ #1 \,\vert\,#2 \}}
\newcommand{\setcl}[2]{\big\{ #1 \,\big\vert\, #2 \big\}}

\newcommand{\calF}{\mathcal F}

\newcommand{\calR}{\mathcal R}

\newcommand{\vvel}{v}
\newcommand{\vpres}{p}
\newcommand{\vvels}{\vvel_0}
\newcommand{\vvelp}{\vvel_\perp}
\newcommand{\vpress}{p_0}

\newcommand{\uvel}{u}
\newcommand{\upres}{\mathfrak p}
\newcommand{\uvels}{\uvel_0}
\newcommand{\uvelp}{\uvel_\perp}

\newcommand{\wvel}{w}
\newcommand{\wpres}{\mathfrak q}
\newcommand{\wvels}{\wvel_0}
\newcommand{\wvelp}{\wvel_\perp}

\newcommand{\proj}{{\mathcal P}}
\newcommand{\projcompl}{{\mathcal P_\bot}}

\newcommand{\wakefct}{\mathscr{s}_\zeta}
\newcommand\BbbGamma{\reflectbox{\rotatebox[origin=c]{180}{$\mathds L$}}}
\newcommand{\fsolvel}{\BbbGamma_\zeta}
\newcommand{\fsolvelss}{\BbbGamma_\zeta^{\mathrm{ste}}}
\newcommand{\fsolvelpp}{\BbbGamma_\zeta^{\mathrm{per}}}
\newcommand{\fsolpres}{{\mathrm P}}
\newcommand{\TT}{{\mathds T}}
\newcommand{\DD}{{\mathds D}}
\newcommand{\II}{{\mathds I}}

\newcommand{\CS}[1]{C^{#1}}

\newcommand{\CSci}{C_0^{\infty}}
\newcommand{\LS}[1]{L^{#1}}
\newcommand{\LSloc}[1]{L^{#1}_{\mathrm{loc}}}
\newcommand{\WS}[2]{W^{#1,#2}}

\title{Representation formulas and far-field behavior
of time-periodic incompressible viscous flow around a translating rigid body}

\author{Thomas Eiter%
\footnote{%
Institute of Mathematics, University of Kassel, Heinrich-Plett Str.~40, 34132 Kassel, Germany}
\footnote{%
Weierstrass Institute for Applied Analysis and Stochastics,
Mohrenstra\ss{}e 39, 10117 Berlin, Germany.
\\
Email: {\texttt{thomas.eiter@wias-berlin.de}}}%
\and
Ana Leonor Silvestre%
\footnote{%
CEMAT and Department of Mathematics, Instituto Superior T\'ecnico, Universidade de
Lisboa, Av. Rovisco Pais 1, 1049-001 Lisboa, Portugal.
\\
Email: {\texttt{ana.silvestre@math.tecnico.ulisboa.pt}}}%
}

\begin{document}
\maketitle

\begin{abstract}
This paper is concerned with integral representations and asymptotic expansions of solutions to the time-periodic incompressible Navier-Stokes equations for fluid flow in the exterior of a rigid body that moves with constant velocity. Using the time-periodic Oseen fundamental solution, we derive representation formulas for solutions with suitable regularity. From these formulas, the decomposition of the velocity component of the fundamental solution into steady-state and purely periodic parts and their detailed decay rate in space, we deduce complete information on the asymptotic structure of the velocity and pressure fields. 
\end{abstract}


\tableofcontents

\section{Introduction}

We are interested in time-periodic incompressible viscous flow around a translating rigid body ${\mathcal O}$. In a reference frame attached to the solid, the velocity $v=v(t,x)$ and pressure $p=p(t,x)$ of the fluid satisfy 
\begin{equation}
\left\{
\begin{aligned}
\displaystyle \partial_t v + v \cdot \nabla v - \nu \Delta v +  \nabla p - \zeta \cdot \nabla v & = f && \textrm{in  } {\mathbb T} \times \Omega,  \\
\displaystyle \nabla \cdot v & = 0 &&  \textrm{in  } {\mathbb T} \times \Omega, \\
\displaystyle  v & = v_b && \textrm{on  }  {\mathbb T} \times \Sigma,  \\
\displaystyle \lim_{|x|\rightarrow \infty} v(t,x) & =0 &&   \textrm{for  } t \in {\mathbb T}, 
\end{aligned}
\right.
\label{periodicproblem}
\end{equation}
where $\Omega \subset {\mathbb R}^3$ is an exterior domain and ${\mathbb T} := \bigslant{{\mathbb R}}{{\mathcal T} {\mathbb Z}}$ is the torus group related to the period ${\mathcal T} >0$ of the motion. The fluid has constant viscosity $\nu >0$ and  is subject to a ${\mathcal T}$-periodic external force $f$. 
The body ${\mathcal O}$ undergoes a translational motion with constant velocity $\zeta \in {\mathbb R}^3\setminus \{0\}$,
and the fluid motion on its surface $\Sigma:=\partial \Omega$ is prescribed via 
${\mathcal T}$-periodic boundary values $v_b$,
which may describe a periodic boundary motion or flux through the boundary. 

The stationary counterpart of \eqref{periodicproblem} is a classical mathematical problem with well-known properties, which have been studied, for example, in \cite{Finn59,RF63,Finn,FINN73,G,Babenko,Farwig_habil,Farwig_statextOseenNSE_92}. In particular, the velocity field exhibits different asymptotic behavior inside and outside a wake region created by the translational motion of $\mathcal O$. This behavior is inherited from the steady Oseen fundamental solution. From the point of view of applications, a precise description of the far-field behavior of the flow can be very useful in numerical studies, for which a bounded computational domain is defined, and consequently, appropriate boundary conditions must be imposed on the outer boundary of the truncated domain. Such boundary conditions should be consistent with the asymptotic decay of solutions, as was shown in \cite{Heuv,DK2011}.

The study of the time-periodic exterior problem \eqref{periodicproblem} is more recent, see \cite{Maremonti91a,Maremonti91b,MaremontiPadula,KozonoNakao1996,GS06,Eiter21, EiterKyed18,EiterKyed17,GaldiKyed18,EiterGaldi21,EiterKyed21} for instance. In this paper, we exploit the mathematical properties of time-periodic Oseen fundamental solution, based on the results of \cite{Ky16,EiterKyed17,EiterKyed18}, wherein a modern mathematical setting was developed for the time-periodic Stokes and Oseen problems and extended to the exterior Navier-Stokes equations. The first step of our analysis is the integral representation obtained in Theorem \ref{thm:repr.nonlin} for $(v,p)$, where we followed \cite{anaJMFM}. Then, employing the decomposition of the velocity component of the fundamental solution into steady-state and purely periodic contributions, we deduce complete asymptotic expansions for the steady and purely periodic parts of the velocity and the pressure fields, the main results being stated in Theorem \ref{thm:asymp.nonlin}.

Comparing our results with the cases $\Omega = {\mathbb R}^3$, see  \cite{Eiter21,EiterKyed17,EiterKyed18}, or with the simpler case $v_b = \zeta$, addressed for example in \cite{EiterGaldi20}, we found that the 
time-periodic boundary data $\vvel_b$
significantly influence the far-field behavior 
of the velocity and of the pressure of the flow. 
More precisely, in Theorem \ref{thm:constflux} we show that, if the total flux through the boundary is time-independent, the purely periodic part of the velocity as well as the pressure show a faster decay rate, the same as for the whole space problem and the corresponding fundamental solutions. 
However, if the flux is time-dependent, 
both quantities decay at a slower rate.

\emph{Structure of the paper.} Section~\ref{sec:notationsfsol} contains the notations for functions spaces, differential operators and distributions that are relevant for the mathematical formulation and analysis of the problem. We also recall the steady and time-periodic Oseen fundamental solutions and some results for their convolution with suitably decaying functions,
and we prepare a general result on estimates for certain convolution integrals. 
In Section~\ref{sec:mainresults}, we present and discuss the main results of the paper. The linearized time-periodic problem is addressed in Section~\ref{sec:linearproblem}. Firstly, we deduce integral representation formulas for the velocity and pressure of the linear problem using the time-periodic Oseen fundamental solution. Subsequently, a detailed analysis of the far-field behavior of solutions is carried out, exploiting the decomposition of the velocity fields into a time-independent part and a time-dependent part with mean value zero. The proofs of our main results are given in Section~\ref{sec:proofs}.

\section{Notations and fundamental solutions}
\label{sec:notationsfsol}

\subsection{Notations}

 The standard basis of ${\mathbb R}^3$ will be denoted by $\{\mathsf{e}_1,\mathsf{e}_2,\mathsf{e}_3\}$. For $x \in {\mathbb R}^3\setminus \{0\}$, we define the unit vector $\hat{x} = x / |x|$, where $\snorm{\cdot}$ denotes the Euclidean norm. 
For $R>0$ 
we set $B_R(x)\coloneqq\setc{y\in\R^3}{\snorm{x-y}<R}$
and $B_R\coloneqq B_R(0)$.
The symbol $\Omega$ always denotes an exterior domain in $\R^3$, that is a domain that is the complement of a compact set in $\R^3$.
The unit outer normal at $\Sigma=\partial\Omega$ is denoted by $n=n(x)$.

The gradient of a vector field $v$ is defined by $(\nabla v)_{ij}=\partial_iv_j=\frac{\partial v_j}{\partial x_i},$ $i,j=1,2,3$, and $\nabla \cdot v=\frac{\partial v_i}{\partial x_i}$ denotes its divergence. 
Here and in what follows, Einstein summation convention is used.
We adopt this notation for the divergence of a tensor field $\mathds F=(\mathds F_{ij})$ and define 
$\nabla \cdot {\mathds F} = \frac{\partial {\mathds F}_{ji}}{\partial x_j}  \mathsf{e}_i$. 
The notation $\TT(v,q)$ is used for the stress tensor $\TT(v,q):= 2 \nu \DD(v)-p \II$, where $\DD(v) = \frac 12 \left( \nabla v + (\nabla v)^\top \right)$ and $\II \in {\mathbb R}^{3 \times 3}$ is the identity matrix.  
 
For $q \in [1,\infty]$ and $m \in {\mathbb N}$, 
$L^q({\mathcal D})$ and $W^{m,q}({\mathcal D})$ denote classical Lebesgue and Sobolev spaces on a domain ${\mathcal D} \subseteq {\mathbb R}^n$,
with norms $\| \cdot \|_{q,{\mathcal D}}$  and $\| \cdot \|_{m,q,{\mathcal D}}$, respectively. If $q \in [1,\infty)$, by $W^{m-\frac{1}{q},q}(\partial{\mathcal D})$ we indicate the trace space for $W^{m,q}(\mathcal D)$-functions on the (sufficiently
smooth) boundary $\partial{\mathcal D}$ of ${\mathcal D}$, equipped with the norm $\| . \|_{m-\frac{1}{q},q,\partial{\mathcal D}}.$ When ${\mathcal D}$ is an exterior domain, it is convenient to consider the homogeneous Sobolev spaces defined by
\[
D^{m,q}({\mathcal D})
:=\{u\in L^1_\loc({\mathcal D})  ;\, D^\alpha u\in L^q({\mathcal D})\;
\mbox{for any multi-index $\alpha$ with $|\alpha|=m$}\}
\]
and equipped with seminorm $|u|_{m,q,{\mathcal D}}=\left( \sum_{|\alpha|=m}\|D^\alpha u\|^q_{q,{\mathcal D}} \right)^{1/q}$, where  $q \in [1,\infty)$. 
The space $D^{m,q}_0({\mathcal D})$ is the closure of $C^\infty_0({\mathcal D})$ with respect to the norm $| \cdot |_{m,q,{\mathcal D}}$. The dual space $(D^{m,q}_0)'({\mathcal D})$ will be denoted by $D^{-m,q'}({\mathcal D})$, where $q':=q/(q-1)$.

If $X$ is a Banach space associated with the space variable, $L^r(\mathbb{T};X)$ denotes the space of all Bochner-measurable functions $u: {\mathbb T} \to X$ such that $\| u\|_{L^r({\mathbb T} ;X)}:=\left(\int_{{\mathbb T} } \|u(t)\|_X^r\ dt \right)^{\frac 1r} = \left( \frac{1}{\mathcal T} \int_{0}^{\mathcal T} \|u(t)\|_X^r\ dt \right)^{\frac 1r} < \infty,$ for $ r \in [1,\infty)$, and $\| u\|_{L^\infty({\mathbb T};X)}:=\esssup_{t \in {\mathbb T} }\|u(t)\|_X < \infty,$ for $r = \infty$. 
Note that we equip $\torus$ with the normalized Lebesgue measure. 
For $r,q\in(1,\infty)$
we further denote the space of boundary traces of functions in $\WS{1}{r}(\torus, \LS{q}(\mathcal D)^3) \cap 
\LS{r}(\torus, \WS{2}{q}(\mathcal D)^3)$
by
\begin{equation}
\label{eq:tracespace}
T_{r,q}(\torus\times\partial\mathcal D)
:=\setcl{ \vvel_b=\vvel|_{\torus\times\partial\mathcal D}}{\vvel\in\WS{1}{r}(\torus, \LS{q}(\mathcal D)^3) \cap \LS{r}(\torus, \WS{2}{q}(\mathcal D)^3)}.
\end{equation}
The space $T_{r,q}(\torus\times\partial\mathcal D)$
can be identified with a real interpolation space,
but this property will not be used in what follows.
Moreover, by $C({\mathbb T} ;X)$ we denote the space of continuous functions from ${\mathbb T}$ to $X$. 

As in \cite{Ky16,EiterKyed17,EiterKyed18}, we will denote the Dirac delta distributions on ${\mathbb R}^3$, $\mathbb T$, and $\mathbb Z$ by $\delta_{{\mathbb R}^3}$, $\delta_{{\mathbb T}}$ and $\delta_{{\mathbb Z}}$, respectively. When studying the whole space problem, it will be formulated in the locally compact abelian group $G := {\mathbb T} \times  {\mathbb R}^3$, and the Dirac delta distribution on $G$, $\delta_G$, will be used to define the fundamental solution to time-periodic problems. 
By $\mathcal S'(\R^3)$, $\mathcal S'(G)$ and $\mathcal S'(\hat G)$, where $\hat G:=\Z\times\R^3$,
we denote the spaces of tempered distributions on $\R^3$, $G$ and $\hat G$,
see~\cite{EiterKyed17} for a precise definition.
Moreover, $1_\torus$ denotes the constant $1$ function on $\torus$.
In the context of the exterior problem, $\delta_\Sigma$ will denote the Dirac delta distribution on $\Sigma$, as defined, for example, in \cite{ergorov}. 

Additional notations will be introduced as they are needed.

\subsection{Steady-state Stokes and Oseen fundamental solutions}

It is well known that
\begin{equation} {\mathrm E}(x)=  \frac{1}{4\pi |x|} \label{la-funda}
\end{equation}
is the fundamental solution of the Laplace operator in $\R^3$, that is, $- \Delta {\mathrm E} = \delta_{{\mathbb R}^3}$ in ${\mathcal S}'({\mathbb R}^3)$.

The fundamental solution of the classical Stokes system in $\R^3$ is the pair $(\BbbGamma_{0}^{\mathrm{ste}}, {\mathrm P}) \in {\mathcal S}'({\mathbb R}^3)^{3 \times 3} \times {\mathcal S}'({\mathbb R}^3)^{3}$ given by (see, for example, \cite{G})
$$
(\BbbGamma_{0}^{\mathrm{ste}},  {\mathrm P})(x) = \left(\frac{1}{8\pi\nu |x|}\left( \II + \hat{x} \otimes \hat{x} \right),\frac{1}{4\pi |x|^2}\hat{x} \right).
$$
In particular, the pressure component satisfies
\begin{equation}
 {\mathrm P}(x) = - \nabla {\mathrm E}(x),
 \label{PE}
\end{equation}
and one directly 
deduces the estimate
\begin{equation}
\forall\alpha\in\N_0^3 \quad \exists C>0 \quad  \forall x\neq 0 :
\quad
\snorml{D^\alpha_x  {\mathrm P}(x)}
\leq C\snorm{x}^{-2-\snorm{\alpha}}.
\label{eq:decay.fsolpres}
\end{equation}

The fundamental solution 
$(\BbbGamma_{\zeta}^{\mathrm{ste}}, {\mathrm P}) \in {\mathcal S}'({\mathbb R}^3)^{3 \times 3} \times {\mathcal S}'({\mathbb R}^3)^{3}$
of the 3D Oseen system has the same pressure part \eqref{PE}, 
and the velocity component is given by (see \cite{G,Farwig_statextOseenNSE_92,anaJMFM})
\begin{equation}
\begin{aligned}
 \BbbGamma_{\zeta}^{\mathrm{ste}}(x)  =  & 
 \frac{1}{4\pi \nu|x|}\exp\left(-\frac{\wakefct(x)}{\nu}\right) \II 
 -  \frac{1-\exp(-\wakefct(x)/\nu)}{8\pi |x| \wakefct(x)}  \left( \II  - \hat{x} \otimes \hat{x} \right)  \medskip \\
&  +  \frac{|\zeta|}{16 \pi}   \frac{1-\exp(-\wakefct(x)/\nu)-\exp(-\wakefct(x)/\nu)\wakefct(x)/\nu}{\wakefct(x)^2}   \left(  \hat{x} + \hat{\zeta}  \right) \otimes \left( \hat{x} + \hat{\zeta} \right),
 \end{aligned}
\label{ose}
\end{equation}
where $\wakefct(x):= \left[|\zeta| |x| + (\zeta\cdot x )\right] / 2.$ 
It is subject to the pointwise estimate (see \cite{Farwig_statextOseenNSE_92})
\begin{equation}
\forall\alpha\in\N_0^3 \quad \forall \varepsilon >0 \quad \exists C>0 \quad  \forall \snorm{x}\geq \varepsilon :
\quad
\snorml{D^\alpha_x \fsolvelss(x)}
\leq C\bb{\snorm{x}\bp{1+\wakefct(x)}}^{-1-\snorm{\alpha}/2}.
\label{eq:decay.fsolvelss}
\end{equation}
Observe that this estimate is anisotropic in space.
In particular, the decay rate depends on the direction $\zeta$,
and reflects the occurrence of a wake region in the flow behind the body. Hence, there is a significant difference between the velocity components of the Stokes and Oseen fundamental solutions.

The following lemma yields estimates on convolutions of $\fsolvelss$
with functions with suitable decay. Throughout this work we will use the notation \begin{equation}
\log_+(r)\coloneqq \max\set{1,\log r} \text{ for } r>0.
\label{log+}
\end{equation}

\begin{lemma}\label{lem:conv.fsolvelss}
Let $g\in\LS{\infty}(\R^3)$ and $A\in[2,\infty)$, $B\in[0,\infty)$, $M>0$  
such that \[ 
\snorm{g(x)}\leq M\np{1+\snorm{x}}^{-A}\np{1+\wakefct(x)}^{-B}.
\]
Then there exists
$ C= C(A,B,\zeta)>0$
with the following properties:
\begin{enumerate}
\item
If $A+\min\set{1,B}>3$, then
\[
\snorml{\snorm{\fsolvelss}\ast g (x)} 
\leq  C M \bb{\np{1+\snorm{x}}\bp{1+\wakefct(x)}}^{-1}.
\]
\item
If $A+\min\set{1,B}>3$ and $A+B\geq7/2$, then
\[
\snorml{\snorm{\nabla\fsolvelss}\ast g (x)} 
\leq  C M \bb{\np{1+\snorm{x}}\bp{1+\wakefct(x)}}^{-3/2}.
\]
\item
If $A+\min\set{1,B}=3$ and $A+B\geq7/2$, then
\[
\snorml{\snorm{\nabla\fsolvelss}\ast g (x)} 
\leq  C M \bb{\np{1+\snorm{x}}\bp{1+\wakefct(x)}}^{-3/2}\log_+\snorm{x}.
\]
\item
If $A+B< 3$, then
\[
\snorml{\snorm{\nabla \fsolvelss}\ast g (x)} 
\leq  C M \np{1+\snorm{x}}^{-\np{A+B}/2}
\bp{1+\wakefct(x)}^{-\np{A+B-1}/2}.
\]
\end{enumerate}
\end{lemma}

\begin{proof}
These are special cases of \cite[Theorems 3.1 and 3.2]{KracmarNovotnyPokorny2001},
see also \cite[Theorem 3.1]{Eiter21}.
\end{proof}

We next provide similar convolution estimates
for $\nabla\fsolpres$ in a particular case.
However, since $\nabla\fsolpres$ is strongly singular at the origin,
we exclude a neighborhood
from the domain of integration.

\begin{lemma}\label{lem:conv.fsolpres}
Let $g\in\LS{\infty}(\Omega)$ and $M>0$
such that
\[ 
\snorm{g(x)}
\leq M\bp{\np{1+\snorm{x}}\np{1+\wakefct(x)}}^{-2}.
\]
Let $R>0$ and $\chi:=\chi_{\R^3\setminus B_R}$ be the characteristic function
of the exterior to the ball $B_R$ centered at the origin.
Then there exists
$ C= C(R,\zeta)>0$
such that
\[
\snorml{\bp{\chi\nabla\fsolpres}\ast g (x)} 
\leq  C M \snorm{x}^{-2}
\min\setl{1,\np{1+\wakefct(x)}^{-2}\log_+\snorm{x}
+\snorm{x}^{-1}\log_+\snorm{x}
+\np{1+\wakefct(x)}^{-1}}.
\]
\end{lemma}

\begin{proof}
We use \eqref{eq:decay.fsolpres}
and the decay of $g$ to estimate
\[
\begin{aligned}
I(x)\coloneqq
\snorml{\bp{\chi\nabla\fsolpres}\ast g (x)}
&\leq
M\int_{\R^3\setminus B_{R}(x)} 
\snorm{x-y}^{-3} \bp{\np{1+\snorm{y}}\np{1+\wakefct(y)}}^{-2} \,\dd y
\\
&\leq 
CM\int_{\R^3} 
\np{1+\snorm{x-y}}^{-3} \bp{\np{1+\snorm{y}}\np{1+\wakefct(y)}}^{-2} \,\dd y.
\end{aligned}
\]
Following the calculations in~\cite[Section~2]{KracmarNovotnyPokorny2001},
we obtain
\[
\begin{aligned}
I(x)
\leq
CM\bp{\snorm{x}^{-2}\np{1+\wakefct(x)}^{-2}\log_+\snorm{x}
+\snorm{x}^{-3}\log_+\snorm{x}
+\snorm{x}^{-2}\np{1+\wakefct(x)}^{-1}}.
\end{aligned}
\]
Alternatively, we can omit the anisotropic contributions 
and directly estimate
\[
I(x)
\leq 
CM\int_{\R^3} 
\np{1+\snorm{x-y}}^{-3} \np{1+\snorm{y}}^{-2} \,\dd y
\leq 
CM\snorm{x}^{-2},
\]
which gives an improved estimate along 
$\setc{x\in\Omega}{\wakefct(x)=0}
=\setc{x\in\Omega}{x=\alpha\zeta, \, \alpha>0}$.
In summary, we conclude the asserted estimate.
\end{proof}

\subsection{Time-periodic fundamental solutions}

In the time-periodic case, the fundamental solutions can be written as solutions to a system of partial differential equations on the group $G=\torus\times\R^3$. Following \cite{Ky16,EiterKyed17,EiterKyed18}, the fundamental solution of the Stokes ($\zeta=0$) or Oseen ($\zeta\not =0$) equations is a pair $(\fsolvel,  {\mathrm Q}) \in {\mathcal S}'(G)^{3 \times 3} \times {\mathcal S}'(G)^{3}$ satisfying
\begin{equation}
\left\{
\begin{aligned}
 \partial_t \fsolvel - \nu \Delta \fsolvel  +  \nabla  {\mathrm Q} - (\zeta \cdot \nabla) \fsolvel &= \II \delta_G &&  \text{in } G,  \\
 \nabla \cdot \fsolvel &= 0 && \text{in } G. 
\end{aligned}
\right.
\label{periodicfundsol}
\end{equation}
The pressure component is given by 
\begin{equation}
 {\mathrm Q} = \delta_{{\mathbb T}} \otimes  {\mathrm P}, 
 \label{pressfs}
\end{equation}
that is, we formally have ${\mathrm Q}(t,x) = \delta_{{\mathbb T}}(t){\mathrm P}(x)$,
where $\mathrm P$ is the pressure of the fundamental solution in the steady cases, given in~\eqref{PE}.
The velocity component can be split in the form
\begin{equation}
\fsolvel = 1_{{\mathbb T}} \otimes \fsolvelss + \fsolvelpp, 
\label{fsveldec}
\end{equation}
with $\fsolvelss$ the velocity steady part of the  fundamental solution, already presented in the previous section, and $\fsolvelpp$ the purely periodic part of $\fsolvel$,
which is given by
$$
\fsolvelpp = {\mathcal F}^{-1}_G \left[  \frac{1-\delta_{\mathbb Z}(k)}{|\xi|^2 + i \left(    \frac{2 \pi k}{\mathcal T}  - \zeta \cdot \xi \right)} \left( \II - \hat{\xi} \otimes \hat{\xi}\right)\right],
$$
where ${\mathcal F}_G: {\mathcal S}'(G) \to {\mathcal S}'(\hat{G})$, $\hat{G}:={\mathbb Z} \times {\mathbb R}^3$, is the Fourier transform on the group $G$.
For details on the definition of the Fourier transform on $G$ and the spaces $\mathcal S'(G)$ and $\mathcal S'(\hat G)$
of tempered distributions,
we refer to \cite{EiterKyed17}.
We recall the following decay properties of $\fsolvelpp$.

\begin{lemma}
\label{lem:decay.fsolvelpp}
Let $q\in[1,\infty)$. Then
\begin{equation}
\forall\alpha\in\N_0^3 \quad  \forall \varepsilon > 0 \quad  \exists C>0 \quad  \forall \snorm{x}\geq \varepsilon :
\quad
\norml{D^\alpha_x \fsolvelpp(\cdot,x)}_{\LS{q}(\torus)}
\leq C\snorm{x}^{-3-\snorm{\alpha}}.
\label{eq:decay.fsolvelpp}
\end{equation}
\end{lemma}
\begin{proof}
See~\cite[Theorem 1.1]{EiterKyed18}.
\end{proof}

In comparison with the pointwise estimate~\eqref{eq:decay.fsolvelss}
for the steady-state fundamental solution,
the estimate~\eqref{eq:decay.fsolvelpp} is of higher order and homogeneous in space.
Therefore, the occurrence of a wake region behind the body
is mainly reflected in the steady-state part of the velocity field.

The following lemma can be used to derive pointwise estimates
for convolutions of the fundamental solution $\fsolvelpp$
and suitably decaying functions.

\begin{lemma}\label{lem:conv.fsolvelpp}
Let $g\in\LS{\infty}(\torus\times\R^3)$ and $A,M\in(0,\infty)$
such that \[ 
\snorm{g(t,x)}\leq M\np{1+\snorm{x}}^{-A}.
\]
Then for any $\varepsilon>0$ there exists 
$C= C(A,\zeta,\per,\varepsilon)>0$
such that
\begin{equation}
\forall \snorm{x}\geq \varepsilon: \quad
\snorml{\left[ \snorm{\nabla\fsolvelpp} \ast_G g \right](t,x)} 
\leq  C M 
\np{1+\snorm{x}}^{-\min\set{A,4}}
\label{est:ConvFundsolpp_Grad}
\end{equation}
and, if $A>3$,
\begin{equation}
\forall \snorm{x}\geq \varepsilon: \quad
\snorml{\left[\snorm{\fsolvelpp}\ast_G g \right](t,x)} 
\leq C M
\np{1+\snorm{x}}^{-3}.
\label{est:ConvFundsolpp}
\end{equation}
\end{lemma}

\begin{proof}
See \cite[Theorem 3.3]{Eiter21}.
\end{proof}

\subsection{Anisotropic convolution estimates}

In the next lemma we study convolutions of functions with anisotropic decay, which is expressed by $\wakefct(x):= \left[|\zeta| |x| + (\zeta\cdot x )\right] / 2$. The characteristic function of the exterior domain $\Omega$ will be denoted by $\chi_\Omega$. 
\begin{lemma}\label{lem:conv.decay}
Let $\zeta\in\R^3\setminus\set{0}$
and $\phi\in\LS{1}(\torus\times(\R^3\setminus\set{0}))$
such that for some $\varepsilon>0$, $\alpha,\beta\geq 0$ and $p\in(1,\infty]$
\[
\exists C>0 \quad  \forall \snorm{x}\geq\varepsilon : \quad \norm{\phi(\cdot,x)}_{\LS{p}(\torus)}
\leq C\snorm{x}^{-\alpha}(1+\wakefct(x))^{-\beta}.
\]
Let  $R>0$ such that $\Sigma \subset B_R$. 
Let $p':=p/(p-1)$, where $\infty':=1$, 
and let $f$ be given by
\begin{enumerate}
\item
$f=\hat f \chi_\Omega$ with 
$\hat f\in \LS{p'}(\torus\times\Omega)$,
$\supp\hat f\subset \torus\times B_{R}$, or
\item
$f=\hat f\delta_\Sigma$ with
$\hat f\in \LS{p'}(\torus\times\Sigma)$,
\end{enumerate}
and define, 
\[
\Lambda(t)\coloneqq 
\begin{cases}
\displaystyle  \int_\Omega \hat f(t,y)\,\dd y & \text{if }f=\hat f\chi_\Omega,
\medskip \\
\displaystyle  \int_\Sigma \hat f(t,y)\,\dd S(y) & \text{if }f=\hat f\delta_\Sigma,
\end{cases} 
\qquad 
\Xi(t)\coloneqq 
\begin{cases}
\displaystyle  \int_\Omega y \hat f(t,y)\,\dd y & \text{if }f=\hat f\chi_\Omega,
\medskip \\
\displaystyle  \int_\Sigma y \hat f(t,y)\,\dd S(y) & \text{if }f=\hat f\delta_\Sigma,
\end{cases} 
\]
and
\[
M\coloneqq
\begin{cases}
\displaystyle \norm{\hat f}_{\LS{p'}(\torus\times\Omega)} & \text{if }f=\hat f\chi_\Omega,
\medskip \\
\displaystyle \norm{\hat f}_{\LS{p'}(\torus\times\Sigma)} & \text{if }f=\hat f\delta_\Sigma.
\end{cases}
\]
For $S>R+\varepsilon$ there exists $C=C(\alpha,\beta,\varepsilon,R,S,\zeta)>0$ such that
\[
\forall \snorm{x}\geq S : \quad \snorm{\nb{\phi\ast_G f}(t,x)}
\leq C\snorm{x}^{-\alpha}(1+\wakefct(x))^{-\beta} M.
\]
If $\psi\in\CS{1}(\torus\times\np{\R^3\setminus\set{0}})$
is such that
$\snorm{\nabla\psi}=\phi$,
then 
there is $C=C(\alpha,\beta,\varepsilon,R,S,\zeta)>0$ such that
\[
\forall \snorm{x}\geq S: \quad 
 \snorml{\nb{\psi\ast_G f}(t,x)-\bb{\psi(\cdot,x) \ast_\torus \Lambda}(t)}
\leq C\snorm{x}^{-\alpha}(1+\wakefct(x))^{-\beta}
M.
\]
If $\xi\in\CS{2}(\torus\times\R^3\setminus\set{0})$
such that
$\snorm{\nabla^2\xi}=\phi$,
then 
there is $C=C(\alpha,\beta,\varepsilon,R,S,\zeta)>0$ such that
\[
\forall \snorm{x}\geq S: \quad \snorml{\nb{\xi \ast_G f}(t,x)-\bb{\xi(\cdot,x) \ast_\torus \Lambda} (t)
+ \bb{\nabla \xi(\cdot,x)\ast_\torus \Xi}(t)}
\leq C\snorm{x}^{-\alpha}(1+\wakefct(x))^{-\beta} M.
\]

In particular, if $\psi$, $\xi$ and $ f$ are time-independent,
then also $\Lambda$ and $\Xi$ are time-independent and the previous estimates reduce to 
\[
\begin{aligned}
\snorm{[\phi\ast_{{\mathbb R}^3} f](x)}
&\leq C\snorm{x}^{-\alpha}(1+\wakefct(x))^{-\beta}
M,
\\
\snorml{\nb{\psi\ast_{{\mathbb R}^3} f}(x)-\Lambda \psi(x)}
&\leq C\snorm{x}^{-\alpha}(1+\wakefct(x))^{-\beta}
M,
\\
\snorml{\nb{\xi\ast_{{\mathbb R}^3} f}(x)-\Lambda\xi(x)
+ \Xi\cdot\nabla \xi(x)}
&\leq C\snorm{x}^{-\alpha}(1+\wakefct(x))^{-\beta}
M.
\end{aligned}
\]
\end{lemma}

\begin{proof}
We only treat the case $ f=\hat f\chi_\Omega$ here. 
For $ f=\hat f\delta_\Sigma$ we can proceed along the same lines 
by replacing volume integrals over $\Omega$ 
with boundary integrals over $\Sigma$.

For $\snorm{x}\geq S>R+\varepsilon\geq R\geq \snorm{y}$ and $\theta\in[0,1]$
we have
\[
\begin{aligned}
\snorm{x-\theta y}
\geq \snorm{x}-\theta \snorm{y}
&\geq (1-\theta R/S)\snorm{x}
\geq(1-R/S)\snorm{x} \geq S-R>\varepsilon,
\\
(1+2\snorm{\zeta} R)\np{1+\wakefct(x-\theta y)}
&\geq 1+2\snorm{\zeta}\snorm{\theta y}+\wakefct(x-\theta y)
\geq 1+\wakefct(x).
\end{aligned}
\]
In particular, this yields 
$\snorm{x}^\alpha\np{1+\wakefct(x)}^\beta
\leq C \snorm{x-\theta y}^\alpha\np{1+\wakefct(x-\theta y)}^\beta$.
Since $\supp\hat f\subset B_R$, 
we thus obtain
\[
\begin{aligned}
\snorml{\nb{\phi\ast_G f}(t,x)}
&\leq C \int_{\Omega} \snorm{x-y}^{-\alpha}(1+\wakefct(x-y))^{-\beta}\,\Bp{\int_\torus\snorm{\hat f(s,y)}^{p'}\,\dd s}^{1/p'}\dd y \\
&\leq C 
\snorm{x}^{-\alpha}(1+\wakefct(x))^{-\beta}
M.
\end{aligned}
\]
Furthermore, we use the fundamental theorem of calculus to deduce
\[
\begin{aligned}
\snorml{\nb{\psi\ast_G f}(t,x)&-\bb{\psi(\cdot,x) \ast_{\mathbb T}} \Lambda (t)}
\\
&=\snormL{\int_\torus\int_\Omega 
\bp{\psi(t-s,x-y)-\psi(t-s,x)}\hat f(s,y)\,\dd y \dd s}
\\
&=\snormL{\int_\torus\int_\Omega
\int_0^1
y\cdot \nabla\psi(t-s,x-\theta y)\hat f(s,y)\,\dd \theta\dd y \dd s}
\\
&\leq C R \int_{\Omega} \int_0^1 \snorm{x-\theta y}^{-\alpha}(1+\wakefct(x-\theta y))^{-\beta}\,\Bp{\int_\torus\snorm{\hat f(s,y)}^{p'}\,\dd s}^{1/p'}\dd\theta\dd y
\\
&\leq C
\snorm{x}^{-\alpha}(1+\wakefct(x))^{-\beta}
M.
\end{aligned}
\]
Similarly, from Taylor's Theorem we deduce
\[
\begin{aligned}
\snorml{\nb{\xi\ast_G f}(t,x) & -\bb{\xi(\cdot,x) \ast_\torus \Lambda (t)} + \bb{\nabla \xi(\cdot,x)\ast_\torus \Xi}(t)} \\
&= \snormL{
\int_{\mathbb T} \int_\Omega \int_0^1 (1-\theta) (y\otimes y):\nabla^2 \xi(x-\theta y) \hat f(s,y)\,\dd\theta \dd y \dd s 
}
\\
&\leq C R^2\int_{\Omega}\int_0^1 \snorm{x-\theta y}^{-\alpha}(1+\wakefct(x-\theta y))^{-\beta} \Bp{\int_\torus\snorm{\hat f(s,y)}^{p'}\,\dd s}^{1/p'} \,\dd\theta\dd y
\\
&\leq C\snorm{x}^{-\alpha}(1+\wakefct(x))^{-\beta}
M,
\end{aligned}
\]
which completes the proof.
\end{proof}

\section{Main results}
\label{sec:mainresults}

In what follows, we always assume that $0\in\R^3\setminus\overline{\Omega}$ and $\zeta\neq 0$.
To formulate suitable integrability properties of solutions, we recall the 
projections $\proj$ and $\projcompl$,
which decompose a function $w\in L^1_\loc(\torus\times\Omega)$ 
into a \emph{steady-state part} $w_0$ and a 
\emph{purely periodic part} $w_\perp$:
\[
w_0(x)\coloneqq
\proj w(x):=\int_\torus w(t,x)\,\dd t,
\qquad
w_\perp(t,x)\coloneqq 
\projcompl w(t,x):= w(t,x)-\proj w(x).
\]

Firstly, we consider weak solutions to~\eqref{periodicproblem}. Since we are dealing with incompressible flows, we will take test functions in $C^\infty_{0,\sigma}(\torus\times\Omega)$, which is the subspace of   $C^\infty_{0}(\torus\times\Omega)^3$ constituted by divergence-free functions (with respect to the space variable).

Assuming $f\in L^2(\torus ; D^{-1,2}(\Omega)^3)$
and $\vvel_b\in L^\infty(\torus; W^{1/2,2}(\Sigma)^3)$, a function $\vvel\in L^1_\loc(\torus\times\Omega)^3$ 
is called a \emph{weak solution} to \eqref{periodicproblem}
if
\begin{enumerate}
\item[i.]
$\nabla\vvel\in L^2(\torus\times\Omega)^{3\times3}$, 
$\proj\vvel\in L^6(\Omega)^3$,
$\projcompl\vvel\in L^{\infty}(\torus;L^2(\Omega)^3)$,
\item[ii.]
$\nabla \cdot \vvel=0$ in $\torus\times\Omega$,
$\vvel=\vvel_b$ on $\torus\times\Sigma$,
\item[iii.]
the identity
\[
\int_\torus\int_\Omega\bb{-\vvel\cdot\partial_t\varphi
+\nabla\vvel:\nabla\varphi
-\zeta\cdot\nabla\vvel\cdot\varphi
+\np{\vvel\cdot\nabla\vvel}\cdot\varphi}\,\dd x\dd t
=\int_\torus\int_\Omega f\cdot\varphi\,\dd x \dd t
\]
holds for all $\varphi\in C^\infty_{0,\sigma}(\torus\times\Omega)$.
\end{enumerate}

For the derivation of the asymptotic decay rates of $v$ and $p$,
we have to ensure increased regularity of weak solutions.
To this end, we additionally assume one of the following conditions:
\begin{align}
\exists\, \kappa,\rho \in (1,\infty) \text{ with }
\frac{2}{\rho}+\frac{3}{\kappa}<1 :  &\quad
\projcompl\vvel\in L^{\rho}(\torus;L^{\kappa}(\Omega)^3),
\label{eq:reg.fct}
\\
\exists\, \kappa,\rho \in (1,\infty) \text{ with }
\frac{2}{\rho}+\frac{3}{\kappa}<2 : &\quad
\nabla\projcompl\vvel\in L^{\rho}(\torus;L^{\kappa}(\Omega)^{3\times 3}).
\label{eq:reg.grad}
\end{align}
Both assumptions increase the regularity of weak solutions
as was shown in~\cite{Eiter23}.

Now, suppose that the exterior domain $\Omega\subset\R^3$ has a  $C^2$-boundary. 
To quantify the regularity of the data,
recall the trace classes $T_{r,q}(\torus\times\Sigma)$ defined in~\eqref{eq:tracespace}.
\begin{theorem}\label{thm:regularity}
Let $f\in L^{r}(\torus;L^q(\Omega)^3)$ 
and $\vvel_b\in T_{r,q}(\torus\times\Sigma)$ for all $q,r\in(1,\infty)$.
Let $\vvel$ be a
weak time-periodic solution to \eqref{periodicproblem}
that satisfies
\eqref{eq:reg.fct} or \eqref{eq:reg.grad}.
Then
\begin{align}
\forall s_2\in(1,\infty), 
\, s_1 & \in(\frac{4}{3},\infty],
\, s_0\in(2,\infty]: \ \
\proj\vvel\in D^{2,s_2}(\Omega)^3\cap D^{1,s_1}(\Omega)^3\cap L^{s_0}(\Omega)^3,\ 
\label{eq:reg.ss}
\\
\forall q,r & \in(1,\infty): \ \
\projcompl\vvel\in W^{1,r}(\torus;L^q(\Omega)^3)\cap L^r(\torus;W^{2,q}(\Omega)^3),
\label{eq:reg.pp}
\end{align}
and there exists a pressure field $\vpres\in L^1_\loc(\torus\times\Omega)$ 
such that
\begin{equation}\label{eq:reg.pres}
\forall s_2\in(1,\infty):\
\proj\vpres\in D^{1,s_2}(\Omega),
\qquad
\forall q,r\in(1,\infty): \ \projcompl\vpres\in L^r(\torus;D^{1,q}(\Omega))
\end{equation}
and \eqref{periodicproblem} 
is satisfied in the strong sense.
\end{theorem}

\begin{proof}
The result was shown in~\cite{Eiter23} for $\vvel_b\in C\np{\torus;C^2(\Sigma)^3}\cap C^1\np{\torus; C(\Sigma)^3}$.
In the proof given there,
the regularity assumptions on the data are derived 
from the time-periodic maximal regularity results in~\cite{EiterKyedShibata23}
for the linearized system,
where boundary data in the class $T_{r,q}(\torus\times\Sigma)$ were studied.
This allows for the stated generalization.
\end{proof}

In this framework, we can derive (implicit) representation formulas 
for time-periodic weak solutions to~\eqref{periodicproblem}. 
In particular, we deduce two different representation formulas
for the velocity field.

\begin{theorem}
\label{thm:repr.nonlin}
Let $f\in L^{r}(\torus;L^q(\Omega)^3)$ 
and $\vvel_b\in T_{r,q}(\torus\times\Sigma)$ for all $q,r\in(1,\infty)$.
Let $\vvel$ be a
weak time-periodic solution to \eqref{periodicproblem}
that satisfies
\eqref{eq:reg.fct} or \eqref{eq:reg.grad}.
Then
\begin{equation}
\begin{aligned}
v(t,x)  
&=  \int_{{\mathbb T}\times \Omega} \fsolvel(t-s,x-y) \nb{ f(s,y)-\vvel(s,y)\cdot\nabla\vvel(s,y)} \,\dd s \dd y
\\
&\qquad 
+ \int_{{\mathbb T}\times \Sigma} \fsolvel(t-s,x-y) \left[ \TT(v,p)(s,y)n (y) + (\zeta \cdot n(y)) v_b(s,y) \right] \,\dd s \dd S(y)
\\
&\qquad
+ \int_{{\mathbb T} \times \Sigma} v_b(t,y)  \cdot  \left[ \DD ( \fsolvel \mathsf{e}_i)(t-s,x-y) n(y) \right] \,\dd s \dd S(y) \,\mathsf{e}_i
\\
&\qquad
 - \int_{\Sigma}  \left[ v_b(t,y) \cdot n(y) \right] {\mathrm P}(x-y) \,\dd S(y),
\end{aligned}
\label{eq:repr.v}
\end{equation}
\begin{equation}
\begin{aligned}
v(t,x)  
&=  \int_{{\mathbb T}\times \Omega} \fsolvel(t-s,x-y) f(s,y) \,\dd s \dd y
\\
& \qquad - \int_{{\mathbb T} \times \Omega} \left[ \nabla ( \fsolvel \mathsf{e}_i )(t-s,x-y) \right] :  \vvel(s,y) \otimes \vvel(s,y)  \,\dd s \dd y  \,\mathsf{e}_i 
\\
&\qquad 
+ \int_{{\mathbb T}\times \Sigma} \fsolvel(t-s,x-y) \left[ \TT(v,p)(s,y)n (y) + (\zeta \cdot n(y)) v_b(s,y) \right] \,\dd s \dd S(y)
\\
&\qquad 
- \int_{{\mathbb T}\times \Sigma} \fsolvel(t-s,x-y) (v_b(s,y) \cdot n(y)) v_b(s,y) \,\dd s \dd S(y)
\\
&\qquad
+ \int_{{\mathbb T} \times \Sigma} v_b(t,y)  \cdot  \left[ \DD ( \fsolvel \mathsf{e}_i ) (t-s,x-y)n(y) \right] \,\dd s \dd S(y) \,\mathsf{e}_i
\\
&\qquad
 - \int_{\Sigma}  \left[ v_b(t,y) \cdot n(y) \right] {\mathrm P}(x-y) \,\dd S(y)
\end{aligned}
\label{eq:repr.v2}
\end{equation}
for all  $(t,x) \in {\mathbb T} \times \Omega$,
and the pressure field $\vpres$, associated to $\vvel$ by Theorem~\ref{thm:regularity},
satisfies
\begin{equation}
\begin{aligned} 
p(t,x)  
&= \vpres_\infty(t) + \int_{\Omega}  {\mathrm P}(x-y) \cdot  \nb{ f(t,y)-\vvel(t,y)\cdot\nabla\vvel(t,y)} \,\dd y
\\ 
&\qquad
+\int_{\Sigma}  {\mathrm P}(x-y) \cdot 
\TT(v,p)(t,y)n(y)\,\dd S(y) 
\\
&\qquad
+ \int_{\Sigma}  {\mathrm P}(x-y) \cdot \bb{
\np{\zeta \cdot n(y)} v_{b}(t,y) +  \np{v_{b}(t,y) \cdot n(y)}\zeta} \,\dd S(y) 
\\
&\qquad
+ \int_\Sigma 2 \nu \,v_{b}(t,y) \cdot \nabla  {\mathrm P}(x-y) n(y) \,\dd S(y) 
\\
&\qquad
+ \int_{\Sigma}  {\mathrm E}(x-y) \left[ \partial_t v_{b}(t,y) \cdot n(y) \right] \,\dd S(y)
\end{aligned}
\label{eq:repr.p}
\end{equation}
for all $(t,x) \in {\mathbb T} \times \Omega$
and some function $p_\infty\in\LS{1}(\torus)$.
\end{theorem}

Based on these representation formulas 
and the decay properties of fundamental solutions, 
we will derive asymptotic expansions for $(v,p)$, in which the two functions 
\begin{equation}
\Phi(t):=\int_{\Sigma} v_{b}(t,y) \cdot n(y)\,\dd S(y), \qquad \Psi(t):=\int_{\Sigma} v_{b}(t,y) \cdot n(y)y\,\dd S(y)
\label{PhiPsi}
\end{equation}
will appear. The scalar function $\Phi$ denotes the total flux of the flow through the boundary~$\Sigma$. To specify the decay of the remainder terms, 
we use the function $\log_+$ defined in~\eqref{log+}. 

\begin{theorem}
\label{thm:asymp.nonlin}
Let $f$ and $\vvel_b$ be as in Theorem~\ref{thm:repr.nonlin} 
with $\supp f$ compact. 
Let $\vvel$ be a weak solution to~\eqref{periodicproblem}
that satisfies~\eqref{eq:reg.fct} or~\eqref{eq:reg.grad}. 
Then the velocity field $\vvel$ satisfies
\begin{equation}
\begin{split}
\vvel(t,x)
&= \bb{\fsolvel (\cdot,x) \ast_\torus {\mathcal F}}(t) 
+  \Phi(t) {\mathrm P}(x)  -  \Psi(t)\cdot  \nabla  {\mathrm P}(x) 
+ {\mathcal R}(t,x),
\end{split}
\label{eq:asexp.vel}
\end{equation}
where 
\[
\begin{aligned}
\calF(t)
&\coloneqq
\int_{\Omega} f(t,y)\,\dd y 
+ \int_{\Sigma}\left[ \TT(v,p)(t,y)n (y) + \bp{\zeta \cdot n(y)-\vvel_b(t,y)\cdot n(y)}\vvel_b(t,y)\right]\dd S(y),
\end{aligned}
\]
and there exists $C>0$ such that 
\begin{align}
\snorm{\mathcal R_0(x)}
&\leq C \bb{\snorm{x}\bp{1+\wakefct(x)}}^{-3/2}\log_+\snorm{x},
\label{eq:rem.vs}
\\
\snorm{\nabla\mathcal R_0(x)}
&\leq C \bb{\snorm{x}\bp{1+\wakefct(x)}}^{-2}\log_+\Bp{\frac{\snorm{x}}{1+\wakefct(x)}},
\label{eq:rem.dvs}
\\
\snorm{\mathcal R_\perp(t,x)}
&\leq C\snorm{x}^{-3},
\label{eq:rem.vp}
\\
\snorm{\nabla\mathcal R_\perp(t,x)}
&\leq C\snorm{x}^{-7/2}\np{1+\wakefct(x)}^{-1/2},
\label{eq:rem.dvp}
\end{align}
for all $(t,x)\in\torus\times\Omega$.
The pressure field $\vpres$,
associated to $\vvel$ by Theorem~\ref{thm:regularity},
satisfies
\begin{equation}
p(t,x) = \vpres_\infty(t) + \Phi'(t) \mathrm E(x) + \bb{{\mathcal F}(t) + \zeta \Phi(t)+\Psi'(t)} \cdot  {\mathrm P}(x)  + {\mathfrak R}(t,x)
\label{eq:asexp.pres}
\end{equation}
for some $p_\infty\in\LS{1}(\torus)$,
and
there exists $C>0$ such that
\begin{equation}
\snorm{\mathfrak R(t,x)}
\leq C\snorm{x}^{-2}
\min\setl{1,\np{1+\wakefct(x)}^{-2}\log_+\snorm{x}
+\snorm{x}^{-1}\log_+\snorm{x}
+\np{1+\wakefct(x)}^{-1}}
\label{eq:rem.p}
\end{equation}
for all $(t,x)\in\torus\times\Omega$ with $\snorm{x}$ sufficiently large.
\end{theorem}

Theorem~\ref{thm:asymp.nonlin} yields new insights into the 
decay properties of the velocity and pressure fields.
First of all,
due to the 
decomposition \eqref{fsveldec} of the velocity component of the fundamental solution, we obtain separate asymptotic expansions for 
the steady-state and purely periodic parts of $\vvel$.

\begin{cor}\label{cor:asymp.nonlin.split}
In the situation of Theorem~\ref{thm:asymp.nonlin},
the velocity field $\vvel=\vvels+\vvelp$ satisfies
\begin{align}
\vvels(x)
&= \fsolvelss (x) {\mathcal F_0} 
+ \Phi_0 {\mathrm P}(x)  
+ {\mathcal R}_0(t,x),
\label{eq:asexp.vels}
\\
\vvelp(t,x)
&= \bb{\fsolvelpp (\cdot,x) \ast_\torus {\mathcal F_\perp}}(t) 
+ \Phi_\perp(t) {\mathrm P}(x) - \Psi_\perp(t)\cdot \nabla {\mathrm P}(x) 
+ {\mathcal R}_\perp(t,x)
\label{eq:asexp.velp}
\end{align}
for all $(t,x)\in\torus\times\Omega$,
where
$\mathcal R=\mathcal R_0+\mathcal R_\perp$ satisfies~\eqref{eq:rem.vs}--\eqref{eq:rem.dvp}.
\end{cor}

Combining the decay properties of the fundamental solutions 
with those of the (faster decaying) remainder terms, 
we deduce the following result.

\begin{cor}
\label{cor:decay.general}
In the situation of Theorem~\ref{thm:asymp.nonlin},
there exists $C>0$ such that
the velocity field $\vvel=\vvels+\vvelp$ and the pressure $\vpres$ satisfy
\begin{align}
\snorm{\vvels(x)}
&\leq C\bb{\snorm{x}\bp{1+\wakefct(x)}}^{-1},
&
\snorm{\nabla\vvels(x)}
&\leq C\bb{\snorm{x}\bp{1+\wakefct(x)}}^{-3/2},
\label{eq:decay.general.vs}
\\
\snorm{\vvelp(t,x)}
&\leq C\snorm{x}^{-2},
&
\snorm{\nabla\vvelp(t,x)}
&\leq C\snorm{x}^{-3},
\label{eq:decay.general.dp}
\\
\snorm{\vpres(t,x)-\vpres_\infty(t)}
&\leq C\snorm{x}^{-1}
\label{eq:decay.general.p}
\end{align}
for all $(t,x)\in\torus\times\Omega$
with $\snorm{x}$ sufficiently large.
\end{cor}

The asymptotic properties of $\vvel$ and $\nabla\vvel$ were already studied in~\cite{Eiter21}
in the case of time-periodic flow in the whole space $\Omega=\R^3$.
Comparing these decay rates 
with those given in Corollary~\ref{cor:decay.general},
we observe that the steady-state parts $\vvels$ and $\nabla\vvels$ decay at the same rate,
while $\vvelp$ and $\nabla\vvelp$ decay faster when $\Omega=\R^3$.
This difference is due to the presence of the boundary,
more precisely, the time-dependent boundary data.
If the total flux $\Phi$ defined in~\eqref{PhiPsi} is time-independent,
that is, if $\Phi\equiv\Phi_0$,
then the asymptotic expansion simplifies and the decay rates for the purely periodic velocity field 
and the pressure increase as follows.

\begin{theorem}
\label{thm:constflux}
In the situation of Theorem~\ref{thm:asymp.nonlin},
assume $\partial_t\Phi \equiv  0$, that is, $\Phi \equiv \Phi_0$.
Then $\vvel$ and $\vpres$ satisfy
\begin{align}
\vvel(t,x)
&= \bb{\fsolvel (\cdot,x) \ast_\torus {\mathcal F}}(t) 
+ {\mathrm P}(x) \Phi_0 -\nabla {\mathrm P}(x) \Psi(t) 
+ {\mathcal R}(t,x),
\label{eq:asexp.vel.constflux}
\\
p(t,x) 
&= \vpres_\infty(t) + \bb{{\mathcal F}(t) + \Phi_0 \zeta  + \Psi'(t)} \cdot  {\mathrm P}(x)  + {\mathfrak R}(t,x)
\label{eq:asexp.pres.constflux}
\end{align}
for some $p_\infty\in\LS{1}(\torus)$
and for $\mathcal F$ as in Theorem~\ref{thm:asymp.nonlin}.
Moreover, $\mathcal R_0$ satisfies~\eqref{eq:rem.vs} and \eqref{eq:rem.dvs},
$\mathcal R_\perp$ satisfies
\begin{align}
\snorm{\mathcal R_\perp(t,x)}
&\leq C\snorm{x}^{-4},
\label{eq:rem.vp.constflux}
\\
\snorm{\nabla\mathcal R_\perp(t,x)}
&\leq C\snorm{x}^{-4},
\label{eq:rem.dvp.constflux}
\end{align}
 and
$\mathfrak R$ is subject to~\eqref{eq:rem.p}.
Additionally, $\nabla\vvelp$ has the asymptotic expansion
\begin{equation}
\partial_j\vvelp(t,x)
=\bb{\partial_j\fsolvelpp (\cdot,x) \ast_\torus \mathcal F_{\perp}^{\mathrm{lin}}}(t) 
-  \Psi_\perp(t)  \partial_j\nabla {\mathrm P}(x)
+ {\mathcal R}_\perp^{\partial_j}(t,x)
\label{eq:asexp.dvp.constflux}
\end{equation}
with
\[
\begin{aligned}
\mathcal F^{\mathrm{lin}}(t)
&\coloneqq \int_{\Omega} f(t,y)\,\dd y 
+ \int_{\Sigma}\left[ \TT(v,p)(t,y)n (y) + \bp{\zeta \cdot n(y)}\vvel_b(t,y)\right]\dd S(y),
\\
\snorml{{\mathcal R}_\perp^{\partial_j}(t,x)}
&\leq C\snorm{x}^{-9/2}\np{1+\wakefct(x)}^{-3/2}
\end{aligned}
\]
for $j=1,2,3$  and all $(t,x)\in\torus\times\Omega$.
In particular, 
there is $C>0$ such that
\begin{align}
\snorm{\vvels(x)}
&\leq C\bb{\snorm{x}\bp{1+\wakefct(x)}}^{-1},
&
\snorm{\nabla\vvels(x)}
&\leq C\bb{\snorm{x}\bp{1+\wakefct(x)}}^{-3/2},
\label{eq:decay.constflux.vs}
\\
\snorm{\vvelp(t,x)}
&\leq C\snorm{x}^{-3},
&
\snorm{\nabla\vvelp(t,x)}
&\leq C\snorm{x}^{-4},
\label{eq:decay.constflux.dp}
\\
\snorm{\vpres(t,x)-\vpres_\infty(t)}
&\leq C\snorm{x}^{-2}
\label{eq:decay.constflux.p}
\end{align}
for all $(t,x)\in\torus\times\Omega$
with $\snorm{x}$ sufficiently large.
\end{theorem}

For constant total flux,
we thus obtain the same decay rates for the velocity field as in the case $\Omega=\R^3$ (without boundary),
see~\cite{Eiter21}.
Moreover, the steady-state part and the purely periodic part of the velocity field 
and the respective gradients as well as the pressure
decay with the same rate as the corresponding fundamental solutions.

Although the decay of
the term $\nabla\mathcal R_\perp$ given in \eqref{eq:rem.dvp.constflux}
is faster than that in~\eqref{eq:rem.dvp},
it is of the same order as the decay of $\nabla\fsolvelpp$.
Therefore, \eqref{eq:asexp.vel.constflux}
does not yield a proper asymptotic expansion for $\nabla\vvelp$
since a leading-order term is not identified.
Instead, if one considers only the linear boundary contributions that appear in
$\mathcal F_\perp^{\mathrm{lin}}$ for the leading term,
then one arrives at~\eqref{eq:asexp.dvp.constflux},
which properly identifies 
the leading term of an asymptotic expansion for $\nabla\vvelp$.

Obviously, 
the assumption of constant total flux is also satisfied 
if we consider a time-independent forcing $f$ and boundary data $\vvel_b$.
In this case we have $\vvel(t,x)=\vvels(x)$,
and from Theorem~\ref{thm:constflux}
we rediscover the well-known asymptotic expansion and decay estimates for
the velocity field associated to steady flow past a body,
see~\cite{Finn59,Finn,Babenko,G}.

In the previous theorems, the assumption that $\snorm{x}$ is large 
is actually not necessary for the decay estimates of the velocity field $\vvel$,
its gradient $\nabla\vvel$
and the corresponding remainders
since these are bounded due to 
the regularity from~\eqref{eq:reg.ss} and~\eqref{eq:reg.pp}
and the embedding theorem from~\cite[Theorem 3]{Eiter23}. 
However, bondedness with respect to time does not follow for the pressure $\vpres$,
so that the assumption cannot be omitted completely.

The proofs of these results will be provided in Section~\ref{sec:proofs}.

\section{The linearized time-periodic problem}
\label{sec:linearproblem}

In the present section, we derive representation formulas and study the asymptotic behavior of the linearized time-periodic problem that is obtained from the Navier-Stokes equations when  the convective term $v \cdot \nabla v = \nabla \cdot (v \otimes v)$ is neglected or replaced with $\nabla \cdot {\mathds F}$ for a tensor field $\mathds F$. At this stage, ${\mathds F}$ is know and satisfies suitable summability properties.

\subsection{Representation formulas for solutions}

Assume $\Omega\subset\R^3$ is an exterior domain
with $C^2$-boundary. For the whole section, let  $f \in L^r({\mathbb T};L^q(\Omega)^3)$, ${\mathds F} \in L^r({\mathbb T};W^{1,q}(\Omega)^{3 \times 3})$ and $\vvel_b\in T_{r,q}(\torus\times\Sigma)$ 
for some $q\in(1,2)$ and $r \in (1,\infty)$,
and consider a solution $(v,p)$ to the (linear) time-periodic Oseen problem 
\begin{equation}
\left\{
\begin{aligned}
\partial_t v - \nu \Delta v +  \nabla p - \zeta \cdot \nabla v &= f - \nabla \cdot {\mathds F} &&  \textrm{in  } {\mathbb T} \times \Omega,   \\
\nabla \cdot v &= 0 &&  \textrm{in  } {\mathbb T} \times \Omega,  \\
v &= v_b &&  \textrm{on  }  {\mathbb T} \times \Sigma,   \\
\lim_{|x|\rightarrow \infty} v(t,x)&=0 &&   \textrm{for  } t \in {\mathbb T} .
\end{aligned}
\right.
\label{periodicproblemlin}
\end{equation}
Note that the existence of a solution $(\vvel,\vpres)$ with $\vvel=\vvels+\vvelp$ and
\[
\begin{aligned}
\vvels&\in D^{2,q}(\Omega)^3\cap L^{2q/(2-q)}(\Omega)^3,
\\
\vvelp&\in W^{1,r}(\torus;L^q(\Omega)^3)\cap L^r(\torus;W^{2,q}(\Omega)^3),
\\
\vpres&\in L^r(\torus;D^{1,q}(\Omega))
\end{aligned}
\]
was shown in~\cite[Theorem 4.7]{EiterShibata24}. 
Moreover, the velocity field $\vvel$ is unique in this class, 
while the pressure $\vpres$ is unique up to addition with a space-independent function.

In order to derive the integral representation formulas for the velocity and pressure, 
we first transform the exterior problem \eqref{periodicproblemlin} to a whole-space problem. To do so, we extend the velocity and pressure to ${\mathbb R}^3\setminus \overline{\Omega}$ in the following way:
\begin{equation}
\widetilde{v}(t,x):= \begin{cases} 0 & \textrm{if } x \in {\mathbb R}^3 \setminus \overline{\Omega},  \\ v(t,x) & \textrm{if } x \in \overline{\Omega}, \end{cases}   \qquad \widetilde{p}(t,x):= \begin{cases} 0 & \textrm{if } x \in {\mathbb R}^3 \setminus \overline{\Omega},  \\ p(t,x) & \textrm{if } x \in \overline \Omega. \end{cases} 
\label{extended}
\end{equation}
In a first step we assume ${\mathds F} \equiv 0$. The next lemma is proved along the same lines as in the steady case \cite{anaJMFM}.

\begin{lemma}
Under the above assumptions on the data
and if $\mathds F\equiv 0$, the extended fields \eqref{extended} satisfy  
\begin{equation}
\left\{\begin{aligned}
\partial_t \widetilde{v} - \nu \Delta \widetilde{v} +  \nabla \widetilde{p} - \zeta \cdot \nabla \widetilde{v} 
&= F + \nu \nabla g 
&& \text{in  } G,  
\\
\nabla \cdot \widetilde{v} 
&= g  
&&  \text{in  } G, 
\\
\lim_{|x|\rightarrow \infty} \widetilde{v}(t,x)
&=0 
&&  \text{for } t \in {\mathbb T},
\end{aligned}\right.
\label{lambdazero}
\end{equation}
where
\begin{equation}
g  : =  - (v_b \cdot n) \delta_{\Sigma} , \label{eq:newforceg}  
\end{equation}
\begin{equation}
\label{defF}
F  : =   \,  \chi_\Omega  f 
  + \TT(v,p)n \delta_{\Sigma} + (\zeta \cdot n)  v_b  \delta_{\Sigma}  + \nu \nabla \cdot \left[(n \otimes v_b  +  v_b \otimes n     )\delta_{\Sigma} \right].
\end{equation}
\end{lemma}
\begin{proof} The distributional gradients of $\widetilde{v} \in {\mathcal S}'(G)^3$ and $\widetilde{p}\in {\mathcal S}'(G)$ are given by
\begin{equation}
\nabla \widetilde{v} = \chi_\Omega \nabla v - n \otimes  v_b  \delta_{\Sigma} , \quad
\nabla \widetilde{p} = \chi_\Omega \nabla p  - p n \delta_{\Sigma},
\label{derdisttvp}
\end{equation}
and for the time derivative, we have
\begin{equation}
\partial_t \widetilde{v} =  \chi_\Omega \partial_t  v.
\label{dertime}
\end{equation}

From \eqref{derdisttvp}, we obtain $
\nabla \cdot \widetilde{v} = \chi_\Omega \nabla \cdot v  - (n \cdot  v_b) \delta_{\Sigma} =  - (v_b \cdot  n) \delta_{\Sigma}=g$ in $G$
with $g$ defined in \eqref{eq:newforceg}. Moreover,
\begin{align}
\zeta \cdot \nabla \widetilde{v} & = \chi_\Omega (\zeta \cdot \nabla v) - (\zeta \cdot n) v_b \delta_{\Sigma}, \label{derzeta} \\
2 \DD(\widetilde{v}) & = \nabla \widetilde{v} + (\nabla \widetilde{v})^\top = 2 \chi_\Omega \DD (v) - v_b \otimes n\delta_{\Sigma} - n \otimes v_b \delta_{\Sigma}, 
\end{align}
and therefore
\[
\begin{aligned}
\displaystyle \nabla \cdot \TT(\widetilde{v} ,\widetilde{p})  = 
  &  \displaystyle \, 2 \nu \nabla \cdot (\chi_\Omega \DD (v) ) - \nu \nabla \cdot \left( v_b \otimes n\delta_{\Sigma} + n \otimes v_b \delta_{\Sigma} \right) - \nabla \cdot \left( \widetilde{p} \, \II \right) \medskip \\
 = &  \displaystyle \,  \chi_\Omega (\nu \Delta v  -\nabla p ) - \TT(v,p)n \delta_{\Sigma} 
-  \nu \nabla \cdot [(n \otimes  v_b  +  v_b \otimes n )\delta_{\Sigma} ].
\end{aligned}
\]
From  $- \nu \Delta \widetilde{v} + \nabla \widetilde{p} =  - \nabla \cdot \TT(\widetilde{v} ,\widetilde{p}) + \nu \nabla (\nabla \cdot \widetilde{v})$ and the previous relations, we also get
\begin{equation}
\begin{aligned}
\displaystyle - \nu \Delta \widetilde{v} + \nabla \widetilde{p}  
 = & \displaystyle \, - \chi_\Omega (\nu \Delta v  -\nabla p ) + \TT(v,p)n \delta_{\Sigma}  + \nu \nabla g  \medskip \\
& \displaystyle \, + \nu \nabla \cdot [(n \otimes v_b -  v_b \otimes n )\delta_{\Sigma} ].
\end{aligned}
\label{derT}
\end{equation}
Adding the identities \eqref{dertime}, \eqref{derT} and \eqref{derzeta} side by side, then replacing the expression $\chi_\Omega (\partial_t v - \nu \Delta v +  \nabla p - \zeta \cdot \nabla v )$ with $\chi_\Omega  f$ yields the first equation of system \eqref{lambdazero} with $F$ given by \eqref{defF}. 
\end{proof}

Let us consider the solution $(\tilde{v},\tilde{p} )$ of whole space problem \eqref{lambdazero}. In general, $\nabla \cdot \tilde{v} = g \not = 0$ and therefore the integral representation of $(\tilde{v},\tilde{p} )$ using the fundamental solution $(\fsolvel , {\mathrm Q})$ cannot be applied directly to \eqref{lambdazero}.  In what follows, the component-wise convolution is taken over ${\mathbb R}^3$ or over $G$. Since, by  \eqref{PE},  
\[
\nabla \cdot ({\mathrm P}\ast_{{\mathbb R}^3} g) = (-\Delta {\mathrm E}) \ast_{{\mathbb R}^3}g = \delta_{{\mathbb R}^3} \ast_{{\mathbb R}^3} g = g,
\]
we will decompose the velocity component of the solution of \eqref{lambdazero} as
\[
\tilde{v} =  (\tilde{v}  -  {\mathrm P} \ast_{{\mathbb R}^3} g) +  {\mathrm P} \ast_{{\mathbb R}^3} g =: \tilde{u} +  {\mathrm P} \ast_{{\mathbb R}^3} g,
\]
where $\tilde{u}$ satisfies $\nabla \cdot \tilde{u} =0$ and 
\[
\begin{aligned}
& \partial_t \tilde{u}  - \nu \Delta \tilde{u} - \zeta \cdot \nabla \widetilde{u}    \medskip \\ 
&\qquad =  \, \partial_t \tilde{v}  - \nu \Delta \tilde{v} - \zeta \cdot \nabla \widetilde{v} -  \partial_t  ({\mathrm P} \ast_{{\mathbb R}^3} g) +  \nu \Delta  ({\mathrm P} \ast_{{\mathbb R}^3} g) +  \zeta \cdot \nabla ({\mathrm P}\ast_{{\mathbb R}^3} g)  \medskip \\
&\qquad = -\nabla \tilde{p}  + F + 2 \nu \nabla g  +  \nabla ({\mathrm P} \ast_{{\mathbb R}^3} (\zeta g)) +   \nabla({\mathrm E} \ast_{{\mathbb R}^3} (\partial_t g)).
\end{aligned}
\]

From 
\[
\left\{
\begin{aligned}
 \displaystyle \partial_t  \tilde{u} - \nu \Delta \tilde{u} - \zeta \cdot \nabla \widetilde{u}  + \nabla \left(\tilde{p}  - 2 \nu  g  - {\mathrm P} \ast_{{\mathbb R}^3}(\zeta g) - {\mathrm E} \ast_{{\mathbb R}^3} (\partial_t g)\right) & = F &&  \text{in } G, \\
 \displaystyle  \nabla \cdot \tilde{u} & = 0 && \text{in } G,
 \end{aligned}
 \right.
\]
we immediately obtain the pressure in terms of a convolution in $G$ with $ {\mathrm Q}$ (recall  \eqref{pressfs}) as
 \[
\tilde{p} -  {\mathrm P} \ast_{{\mathbb R}^3} (\zeta g) - 2 \nu  g  - {\mathrm E} \ast_{{\mathbb R}^3} (\partial_t g) 
= {\mathrm Q} \ast_{G} F +  p_\infty
= \mathrm P \ast_{\R^3} F +  p_\infty,
\]
where $\vpres_\infty\in\LSloc{1}(G)$
such that $\vpres_\infty(t,\cdot)$ is a polynomial for each $t\in\torus$.
This can readily be verified via Fourier transform, see~\cite[Lemma 5.2.5]{Eiter20_diss}.
In virtue of the decay of the convolutions and the integrability properties of $\nabla\vpres$ and the convolutions,
these polynomials must be constant, so that $ p_\infty\in \LSloc{1}(\torus)=\LS{1}(\torus)$.
Therefore,
\begin{equation}
\tilde{p} = p_\infty + {\mathrm P} \ast_{\R^3}  \bb{F +\zeta g} + {\mathrm E} \ast_{{\mathbb R}^3} (\partial_t g)  + 2 \nu  g,  \label{presssurewhole}
\end{equation}
where $g$ and $F$ are given by \eqref{eq:newforceg} and \eqref{defF}, respectively, and $
\partial_t g  =  - (\partial_t v_b \cdot n ) \delta_{\Sigma} $.

The velocity component of a solution to the time-periodic whole-space problem \eqref{lambdazero} 
is unique up to addition by a polynomial in space, 
see~\cite[Lemma 5.2.5]{Eiter20_diss},
so that the decay of the convolution and the integrability properties of $\vvel$
yield the identity $ \tilde{u}  = \fsolvel \ast_{G} F $,
that is,
\[
\begin{aligned}
 \displaystyle  \tilde{u}
   = {} &  \displaystyle  (1_{{\mathbb T}} \otimes \fsolvelss) \ast_{G}   \left[ f \chi_\Omega + \TT(v,p)n \delta_{\Sigma} + (\zeta \cdot n)  v_b \delta_{\Sigma} \right] \medskip \\  
& \displaystyle + \nu  (1_{{\mathbb T}} \otimes \fsolvelss)  \ast_{G}  \nabla \cdot \left[  (n \otimes v_b  +  v_b \otimes n  )\delta_{\Sigma} \right]  
\medskip \\
& +  \fsolvelpp \ast_{G}   \left[  f \chi_\Omega  + \TT(v,p)n \delta_{\Sigma} + (\zeta \cdot n)  v_b \delta_{\Sigma} \right] \medskip \\  
& \displaystyle +  \nu  \fsolvelpp \ast_{G}  \nabla \cdot \left[ (n \otimes v_b  +    v_b \otimes n   )\delta_{\Sigma}   \right] , 
\end{aligned}
\]
where the first two convolutions, which are time-independent, can be written as
\[
\begin{aligned}
& (1_{{\mathbb T}} \otimes \fsolvelss) \ast_{G}   \left[  f \chi_\Omega   + \TT(v,p)n \delta_{\Sigma} + (\zeta \cdot n)  v_b \delta_{\Sigma} \right] (x) \\
&\qquad =  \int_{\mathbb T}  \int_{\Omega} \fsolvelss(x-y)  f(t,y) \, \dd y \dd t  \\
&\qquad\qquad + \int_{\mathbb T} \int_{\Sigma} \fsolvelss(x-y) \left[ \TT(v,p)(t,y)n (y) + (\zeta \cdot n)(y) v_b(t,y)  \right] \, \dd S(y) \dd t \\
&\qquad =  \int_{\Omega} \fsolvelss(x-y)\left[  {\mathcal P}f(y)  \right] \, \dd y  \\
&\qquad\qquad + \int_{\Sigma} \fsolvelss(x-y) \left[ \mathcal P \TT(v,p) n (y) + (\zeta \cdot n)(y) {\mathcal P} v_b(y) \right] \, \dd S(y)   \\
&\qquad = : {\widetilde u}_0^{(1)}(x)  +  {\widetilde u}_0^{(2)}(x), 
\end{aligned}
\]
and
\[
\begin{aligned}
& \nu (1_{{\mathbb T}} \otimes \fsolvelss)  \ast_{G}  \nabla \cdot \left[(n \otimes v_b  +   v_b \otimes n     )\delta_{\Sigma} \right] (x) \\
&\qquad = - \nu \int_{\mathbb T}  \int_{\Sigma}  v_b(t,y) \cdot   \left[ \nabla_y\left( ( \fsolvelss )(x - y)\mathsf{e}_i\right) n(y)\right] \mathsf{e}_i \, \dd S(y)  \dd t \\
 & \qquad\qquad - \nu \int_{\mathbb T}  \int_{\Sigma} n(y) \cdot   \left[ \nabla_y\left( \fsolvelss(x  - y)\mathsf{e}_i\right) v_b(t,y) \right] \mathsf{e}_i \, \dd S(y)  \dd t \\
&\qquad = - \int_{\mathbb T}\int_{\Sigma} v_b(t,y) \cdot {\mathbb D} ( \fsolvelss (x-y)\mathsf{e}_i)n(y)  \, \dd S(y)  \dd t \, \mathsf{e}_i  \\
&\qquad = - \int_{\Sigma} {\mathcal P} v_b(y) \cdot {\mathbb D} ( \fsolvelss (x-y)\mathsf{e}_i)n(y)  \, \dd S(y)  \dd t \, \mathsf{e}_i   = : {\widetilde u}_0^{(3)}(x).
\end{aligned}
\]
The time-periodic terms can be written as
\[
\begin{aligned}
&\fsolvelpp \ast_{G}  \left[ f \chi_\Omega  + \TT(v,p)n \delta_{\Sigma} + (\zeta \cdot n)  v_b \delta_{\Sigma} \right] \\
&\qquad =   \int_{{\mathbb T}\times \Omega} \fsolvelpp (t-s,x-y) f(s,y) \, \dd s \dd y \\
& \qquad\qquad + \int_{{\mathbb T} \times \Sigma}  \fsolvelpp (t-s,x-y) \left[ \TT(v,p)(s,y)n (y) + (\zeta \cdot n)(y) v_b(s,y) \right]  \, \dd s \dd S(y)  \\
&\qquad = : {\widetilde u}_\perp^{(1)}(t,x)  +  {\widetilde u}_\perp^{(2)}(t,x), 
\end{aligned}
\]
and
\[
\begin{aligned}
& \nu \fsolvelpp \ast_{G}  \nabla \cdot \left[ (n \otimes v_b  +  v_b  \otimes n )\delta_{\Sigma} \right] \\
&\qquad =  -  \int_{\mathbb T} \int_{\Sigma} v_b(t,y)  \cdot {\mathbb D} ( \fsolvelpp (t-s,x-y)\mathsf{e}_i )n(y) \, \dd S(y) \dd s \, \mathsf{e}_i  = :\displaystyle {\widetilde u}_\perp^{(3)}(t,x).
\end{aligned}
\]
For the velocity $v$ in the exterior domain, the previous results yield the formula
$$
v(t,x) = \tilde{u} (t,x) -  \int_{\Sigma} v_b(t,y) \cdot n(y) {\mathrm P}(x-y) \, \dd S(y), \quad (t,x) \in {\mathbb T} \times \Omega.
$$

Concerning the pressure term, we restrict $\tilde{p}$ from~\eqref{presssurewhole} to $\Omega$ using the fact that the support of $g$ is contained in $\Sigma$.
For $(t,x) \in {\mathbb T} \times \Omega$ we thus obtain
\[
\begin{aligned}
p(t,x) 
&=  p_\infty(t) + \nb{{\mathrm P} \ast_{{\mathbb R}^3} ( F+\zeta g)}(t,x)  +  \nb{ {\mathrm E} \ast_{{\mathbb R}^3} (\partial_t g)} (t,x) \\
& = p_\infty(t) + \int_{\Omega} f(t,y) (t,y) \cdot {\mathrm P}(x-y) \dd y 
\\
& \qquad
+ \int_{\Sigma} {\mathrm P}(x-y) \cdot  \TT(v,p)(t,y)n(y) \, \dd S(y)  
+ \int_{\Sigma} \zeta \cdot n(y) {\mathrm P}(x-y) \cdot v_{b}(t,y) \, \dd S(y)   \\
& \qquad + 2 \nu \int_{\Sigma} v_{b}(t,y) \cdot \nabla {\mathrm P}(x-y) n(y) \, \dd S(y) 
+ \int_{\Sigma} \zeta \cdot {\mathrm P}(x-y) v_{b}(t,y) \cdot n(y) \, \dd S(y)    \\
& \qquad + \int_{\Sigma} {\mathrm E}(x-y) \partial_t v_{b}(t,y) \cdot n(y) \, \dd S(y)  \\
& =:   p_\infty(t) + \widetilde{p}(t,x)  + \int_{\Sigma} {\mathrm E}(x-y) \partial_t v_{b}(t,y) \cdot n(y) \, \dd S(y).
\end{aligned}
\]

With the above notations for ${\widetilde u}_0^{(i)}(x)$, $ {\widetilde u}_\perp^{(i)}(t,x)$, $i=1,2,3$, and $\widetilde{p}(t,x)$,  we thus proved:
\begin{theorem} \label{thm:representation}
In the case ${\mathds F}\equiv 0$, the solution $(v,p)$ of problem \eqref{periodicproblemlin} admits the representation 
\begin{align}
\begin{split}
v(t,x)  
& =  {\widetilde u}_0^{(1)}(x)  +  {\widetilde u}_0^{(2)}(x)  +  {\widetilde u}_0^{(3)}(x) \medskip \\ 
&  \quad +  {\widetilde u}_\perp^{(1)}(t,x)  +  {\widetilde u}_\perp^{(2)}(t,x)  +  {\widetilde u}_\perp^{(3)}(t,x) \medskip \\
&  \quad  - \int_{\Sigma}  v_b(t,y) \cdot n(y) {\mathrm P} (x-y) \,\dd S(y),
 \end{split}
\label{eq:repr.v.lin}
\\
\begin{split}
p(t,x)  &= p_\infty(t) + \widetilde{p}(t,x)  + \int_{\Sigma} {\mathrm E}(x-y) \partial_t v_{b}(t,y) \cdot n(y) \,\dd S(y)
\label{eq:repr.p.lin}
\end{split}
\end{align}
for all $(t,x)\in\torus\times\Omega$ and some $p_\infty\in\LS{1}(\torus)$.
\end{theorem}

Now we address the case ${\mathds F} \not \equiv 0$. Later on, the tensor field ${\mathds F}$ will be replaced with the nonlinear term $v \otimes v$. 

\begin{lemma}
\label{lemmaF}
Consider the problem \eqref{periodicproblemlin} with $f\equiv 0$ and $v_b\equiv 0$. Under the stated assumptions on ${\mathds F}$, the extended fields $\widetilde{v}$ and  $\widetilde{p}$ defined in \eqref{extended} satisfy
\begin{equation}
\left\{
\begin{aligned}
 \partial_t \widetilde{v} - \nu \Delta \widetilde{v} +  \nabla \widetilde{p} - \zeta \cdot \nabla \widetilde{v} &= \TT(v,p)n \delta_{\Sigma}  -  \nabla \cdot  (\chi_\Omega {\mathds F})  -  (n^\top {\mathds F}) \delta_{\Sigma}  
&& \text{in  } G,   \\
 \nabla \cdot \widetilde{v} &= 0  
&&  \text{in  } G,  \\
 \lim_{|x|\rightarrow \infty} \widetilde{v}(t,x) &=0 
&&  \text{for } t \in {\mathbb T}.
\end{aligned}
\right.
\label{sistF}
\end{equation}
\end{lemma}

\begin{proof} In this case, the distributional gradients of $\widetilde{v}$ and $\widetilde{p}$ are simply given by
\[
\nabla \widetilde{v} = \chi_\Omega \nabla v, \quad
\nabla \widetilde{p} = \chi_\Omega \nabla p  - p n \delta_{\Sigma},
\]
so that from the proof of Lemma \ref{lambdazero} with $f\equiv 0$ and $v_b\equiv 0$, we get $\nabla \cdot \widetilde{v} = \chi_\Omega \nabla \cdot v = 0$ in $G$ and 
\[
\partial_t \widetilde{v} - \nu \Delta \widetilde{v} +  \nabla \widetilde{p} - \zeta \cdot \nabla \widetilde{v} = \chi_\Omega (\partial_t v - \nu \Delta v +  \nabla p - \zeta \cdot \nabla v ) +  \TT(v,p)n \delta_{\Sigma}  \text{ in  } G. 
\]
Replacing $\chi_\Omega (\partial_t v - \nu \Delta v +  \nabla p - \zeta \cdot \nabla v )$ with $- \chi_\Omega   (\nabla \cdot {\mathds F}) $ and using the relation for the distributional divergence
\[\nabla \cdot (\chi_\Omega {\mathds F})  = \chi_\Omega \nabla \cdot {\mathds F}  - (n^\top {\mathds F}) \delta_{\Sigma}\] 
yields the first equation of  \eqref{sistF}.    
\end{proof}

As before, from the whole space problem \eqref{sistF}, we derive representation formulas.
\begin{theorem} \label{thm:representationF}
If ${\mathds F} \in L^\infty({\mathbb T} \times {\mathbb R}^3)^{3 \times 3}$ and $|{\mathds F}(t,x)|\leq M(1+ |x|)^{-A}$, for some $A \geq 2$, then the solution $(v,p)$ to problem \eqref{periodicproblemlin} with $v_b\equiv 0$ and $f\equiv 0$ admits the representation 
\begin{align}
\begin{split}
v(t,x)  
= & \displaystyle  \int_{\Sigma} \fsolvelss(x-y) \int_{\mathbb T}  \left[ \TT(v,p)(t,y)n (y) - n(y)^\top {\mathds F}(t,y)  \right] \dd t \, \dd S(y)  \medskip \\
& \displaystyle  -  \int_{\Omega}  \nabla \left( \fsolvelss  \mathsf{e}_i \right)  (x  - y)  :  \int_{\mathbb T}  {\mathds F} (t,y) \dd t  \, \dd y \, \mathsf{e}_i  \medskip \\
&  \displaystyle  + \int_{\mathbb T}  \int_{\Sigma}  \fsolvelpp(t-s,x-y)  \left[ \TT(v,p)(t,y)n (y) - n(y)^\top {\mathds F}(t,y)  \right]  \, \dd S(y) \dd t \medskip \\
& \displaystyle -  \int_{\mathbb T}  \int_{\Omega}   \nabla \left(\fsolvelpp \mathsf{e}_i\right) (t-s,x  - y)  :  {\mathds F} (t,y)  \, \dd y \dd t \, \mathsf{e}_i, 
\end{split}
\label{eq:repr.v.linF}
\\
\begin{split}
p(t,x)   =  
& \displaystyle \, p_\infty(t) + \int_{\Sigma} {\mathrm P}(x-y) \cdot  \TT(v,p)n(y) \, \dd S(y)  \medskip 
- \int_{\Omega} (\nabla \cdot {\mathds F})(t,y) \cdot {\mathrm P}(x-y) \dd y  
\end{split}
\label{eq:repr.p.linF}
\end{align}
for all $(t,x)\in\torus\times\Omega$ and some $p_\infty\in\LS{1}(\torus)$.
\end{theorem}
\begin{proof}
Consider the problem \eqref{sistF}. Taking into account that $\nabla \cdot \widetilde{v} = 0$ in $G$, we obtain 
from the same uniqueness argument as above that
\[
\begin{aligned}
\tilde{p} 
&=  p_\infty + {\mathrm Q} \ast_{G} \left[ \TT(v,p)n \delta_{\Sigma}  -  \nabla \cdot  (\chi_\Omega {\mathds F})  -  (n^\top {\mathds F}) \delta_{\Sigma} \right]  \\
&= p_\infty + {\mathrm P} \ast_{{\mathbb R}^3} \left[ \TT(v,p)n \delta_{\Sigma}  -  (n^\top {\mathds F}) \delta_{\Sigma} \right] - (\nabla  {\mathrm P} ) \ast_{{\mathbb R}^3}  (\chi_\Omega {\mathds F}) 
\end{aligned}
\]
for some $p_\infty\in\LS{1}(\torus)$,
and
  \[
\begin{aligned}
\tilde{v} 
& = \fsolvel \ast_{G} \left[ \TT(v,p)n \delta_{\Sigma}  -  \nabla \cdot  (\chi_\Omega {\mathds F})  -  (n^\top{\mathds F}) \delta_{\Sigma} \right]  \\
& = (1_{{\mathbb T}} \otimes \fsolvelss) \ast_{G}   \left[  \TT(v,p)n \delta_{\Sigma}  -  (n^\top{\mathds F}) \delta_{\Sigma}   \right] - \left[\left(1_{{\mathbb T}} \otimes \nabla (\fsolvelss \mathsf{e}_i) \right) \ast_{G} \left( \chi_\Omega {\mathds F} \right) \right] \mathsf{e}_i \\  
&\qquad  +  \fsolvelpp \ast_{G}   \left[ \TT(v,p)n \delta_{\Sigma}  -  (n^\top{\mathds F}) \delta_{\Sigma}  \right] -  \left[ \nabla ( \fsolvelpp \mathsf{e}_i ) \ast_{G}   \left(\chi_\Omega {\mathds F} \right) \right]  \mathsf{e}_i.
\end{aligned}
\]
From here onwards, the arguments are similar to those of Theorem \ref{thm:representation} and lead to the representations~\eqref{eq:repr.v.linF} and~\eqref{eq:repr.p.linF}.
\end{proof}
\subsection{Asymptotic expansion of the velocity and pressure fields}

First, we consider the case ${\mathds F} \equiv 0$. Here, we use the functions $\Phi$ and $\Psi$
defined in~\eqref{PhiPsi}, which depend on the boundary data $\vvel_b$.

\begin{theorem} \label{thm:asymp.lin}
Let $(v,p)$ be the solution of system \eqref{periodicproblemlin} with ${\mathds F} \equiv 0$,
and such that $f$ has compact support. The asymptotic expansion at infinity of the velocity field 
$\vvel$ is of the form
\begin{equation}
\begin{split}
\vvel(t,x)
&= \bb{\fsolvel (\cdot,x) \ast_{\mathbb T} \mathcal F^{\mathrm{lin}}}(t) 
- \Phi(t)  {\mathrm P}(x) + \Psi(t) \cdot \nabla {\mathrm P}(x) 
+ {\mathcal R}(t,x),
\end{split}
\label{eq:asexp.v}
\end{equation}
where 
\begin{equation}
\mathcal F^{\mathrm{lin}}(t)
\coloneqq \int_{\Omega} f(t,y)\,\dd y +  \int_{\Sigma}\left[ \TT(\vvel,\vpres)(t,y)n (y) + (\zeta \cdot n)(y) v_b(t,y) \right]\dd S(y),
\end{equation}
and there is $C>0$ such that
\begin{equation}\label{eq:remainder.v}
\begin{aligned}
\snorm{{D_x^\alpha\mathcal R}_{0}(x)}
\leq C\bb{\snorm{x}\,\np{1+\wakefct(x)}}^{-3/2-\snorm{\alpha}/2},
\qquad
\snorm{{D_x^\beta\mathcal R}_{\perp}(t,x)}
\leq C\snorm{x}^{-4-\snorm{\beta}}
\end{aligned}
\end{equation} 
for all $\alpha,\beta\in\N_0^3$ with $\snorm{\beta}\leq 2$,
and all $(t,x)\in\torus\times\Omega$
with $\snorm{x}$ sufficiently large.
In particular, asymptotic expansions for the steady-state part $\vvels$ and the purely periodic part $\vvelp$ of $\vvel$ are given by
\begin{align}
\vvels(x) 
&= \fsolvelss (x) \mathcal F^{\mathrm{lin}}_{0} 
- \Phi_0 {\mathrm P}(x) 
+ {\mathcal R}_{0}(x),
\label{eq:asexp.vs}
\\
\begin{split}
\vvelp(t,x)
&= \bb{\fsolvelpp (\cdot,x) \ast_{\mathbb T} \mathcal F^{\mathrm{lin}}_{\perp}}(t) 
-  \Phi_\perp(t) {\mathrm P}(x) + \Psi_\perp(t) \cdot \nabla {\mathrm P}(x) 
+ {\mathcal R}_{\perp}(t,x).
\end{split}
\label{eq:asexp.vp}
\end{align}

The asymptotic expansion at infinity of the pressure is of the form
\begin{equation}
p(t,x) = p_\infty(t) + \Phi'(t) {\mathrm E}(x) + \bb{\mathcal F^{\mathrm{lin}}(t) + \zeta\Phi(t) - \Psi'(t)} \cdot  {\mathrm P}(x)  + {\mathfrak R}(t,x)
\label{eq:asexp.p}
\end{equation}
for some $p_\infty\in\LS{1}(\torus)$
and
\begin{equation}\label{eq:remainder.p}
\snorm{D_x^\alpha{\mathfrak R_0}(x)}
\leq C\snorm{x}^{-3-\snorm{\alpha}},
\qquad
\snorm{D_x^\beta{\mathfrak R_\perp}(t,x)}
\leq C\snorm{x}^{-3-\snorm{\beta}}
\end{equation}
for all $\alpha,\beta\in\N_0^3$ with $\snorm{\beta}\leq 1$,
and all $(t,x)\in\torus\times\Omega$ with $\snorm{x}$ sufficiently large.
\label{TheorMain1}
\end{theorem}

\begin{proof}
We derive~\eqref{eq:asexp.vs} and~\eqref{eq:asexp.vp} 
from the representation formula for $\vvel$ in Theorem~\ref{thm:representation}.
For the steady-state part $\vvels=\proj\vvel$
we have
\[
\begin{aligned}
\vvels(x)=
{\widetilde u}_0^{(1)}(x)  +  {\widetilde u}_0^{(2)}(x)  +  {\widetilde u}_0^{(3)}(x) - \int_{\Sigma} \proj v_b(y) \cdot n(y) {\mathrm P}(x-y) \,\dd S(y).
\end{aligned}
\]
Since 
$\snorm{\nabla\fsolvelss(x)}\leq C\bp{\snorm{x}\np{1+\wakefct(x)}}^{-3/2}$
for $\snorm{x}\geq \varepsilon$,
from Lemma~\ref{lem:conv.decay}
we obtain
\[
\begin{aligned}
\snorml{{\widetilde u}_0^{(1)}(x)  +  {\widetilde u}_0^{(2)}(x) 
 - \fsolvelss(x)\mathcal F^{\mathrm{lin}}_0}
\leq C\bp{\snorm{x}\np{1+\wakefct(x)}}^{-3/2}
\end{aligned}
\]
for $\snorm{x}$ sufficiently large.
Lemma~\ref{lem:conv.decay} further yields
\[
\begin{aligned}
\snorml{{\widetilde u}_0^{(3)}(x)}
\leq C \bp{\snorm{x}\np{1+\wakefct(x)}}^{-3/2}
\end{aligned}
\]
for $\snorm{x}$ sufficiently large.
Similarly, using Lemma~\ref{lem:conv.decay} and
the identity $\snorm{\nabla {\mathrm P}(x)}=C\snorm{x}^{-3}$, 
we deduce
\[
\begin{aligned}
\snormL{\int_{\Sigma} \proj v_b(y) \cdot n(y) {\mathrm P}(x-y) \,\dd S(y) 
- {\mathrm P}(x) \Phi_0}
\leq C\snorm{x}^{-3}
\end{aligned}
\]
for $\snorm{x}$ sufficiently large.
Collecting the last three estimates, we conclude~\eqref{eq:asexp.vs}
with the asserted estimate of the remainder $\mathcal R_{0}$. 
Since the purely periodic part $\vvelp=\projcompl\vvel$
satisfies
\[
\begin{aligned}
\vvelp(t,x)=
{\widetilde u}_\perp^{(1)}(t,x)  +  {\widetilde u}_\perp^{(2)}(t,x)  +  {\widetilde u}_\perp^{(3)}(t,x) - \int_{\Sigma} \projcompl v_b(t,y) \cdot n(y) {\mathrm P}(x-y) \,\dd S(y),
\end{aligned}
\]
we obtain in the same way that
\[
\begin{aligned}
\snorml{{\widetilde u}_\perp^{(1)}(t,x)  +  {\widetilde u}^{(2)}_\perp(t,x) 
-\bb{\fsolvelpp(\cdot,x)\ast_{\mathbb T}\mathcal F_{\vvel,\perp}}(t)}
&\leq C\snorm{x}^{-4},
\\
\snorml{{\widetilde u}_\perp^{(3)}(t,x)}
&\leq C \snorm{x}^{-4},
\\
\snormL{\int_{\Sigma} \projcompl v_b(t,y) \cdot n(y) {\mathrm P}(x-y) \,\dd S(y) 
- {\mathrm P}(x) \Phi_\perp(t)
+\nabla {\mathrm P}(x) \Psi_\perp(t) }
&\leq C\snorm{x}^{-4}
\end{aligned}
\]
for large $\snorm{x}$.
Summarizing, we obtain~\eqref{eq:asexp.vp} with 
$\snorm{{\mathcal R}_{\perp}(t,x)}
\leq C\snorm{x}^{-4}$ for large $\snorm{x}$.

Now let us turn to the asymptotic expansion~\eqref{eq:asexp.p} of the pressure.
We define
\[
\begin{aligned}
\vpres^{(1)}(t,x)
& =  \int_{\Omega} f(t,x) \cdot {\mathrm P}(x-y) \,\dd y, 
\\
\vpres^{(2)}(t,x)
& = \int_{\Sigma} {\mathrm P}(x-y) \cdot  \TT(v,p)n(y) \,\dd S_y + \int_{\Sigma} \zeta \cdot n(y) {\mathrm P}(x-y) \cdot v_{b}(t,y) \,\dd S_y
\\
&   \quad
+ \int_{\Sigma} \zeta \cdot {\mathrm P}(x-y) v_{b}(t,y) \cdot n(y)\,\dd S_y, 
\\
\vpres^{(3)}(t,x)
& = \int_{\Sigma} {\mathrm E}(x-y) \partial_t v_{b}(t,y) \cdot n(y)\,\dd S_y, 
\end{aligned}
\]
so that $\vpres(t,x)=\vpres^{(1)}(t,x)+\vpres^{(2)}(t,x)+\vpres^{(3)}(t,x)+p_\infty(t)$.
Arguing as above, 
we use
$\snorm{\nabla {\mathrm P}(x)}\leq C\snorm{x}^{-3}$
for $x\neq0$
to obtain
\[
\begin{aligned}
\snormL{\vpres^{(1)}(t,x)- {\mathrm P}(x) \cdot \int_\Omega f(t,y)\,\dd y }
&\leq 
C\snorm{x}^{-3},
\\
\snormL{\vpres^{(2)}(t,x)
-{\mathrm P}(x)\cdot\int_{\Sigma}\left[ \TT(v,p)(t,y)n (y) + (\zeta \cdot n)(y) v_b(t,y) \right]\,\dd S_y
- \Phi(t)  {\mathrm P}(x)\cdot\zeta }
&\leq C\snorm{x}^{-3},
\\
\snormL{\vpres^{(3)}(t,x)- \Phi'(t) {\mathrm E}(x) + \Psi'(t)\cdot {\mathrm P}(x) }
&\leq 
C\snorm{x}^{-3}
\end{aligned}
\]
for $\snorm{x}$ sufficiently large.
Collecting these estimates,
we arrive at \eqref{eq:asexp.p} 
with remainder term $\mathfrak R$ satisfying~\eqref{eq:remainder.p} for $\alpha=\beta=0$.

To obtain \eqref{eq:remainder.v} for $\snorm{\alpha},\snorm{\beta}\leq2$ and
\eqref{eq:remainder.p} for $\snorm{\alpha},\snorm{\beta}\leq 1$,
we can argue in the same way.
For example, for the pressure $p$ we thus have
\[
\partial_j\mathfrak R(t,x)
=\partial_j p(t,x)-\Phi'(t) \partial_j E(x) 
- \bb{\mathcal F^{\mathrm{lin}}(t)+\zeta\Phi(t)-\Psi'(t)} \cdot \partial_j {\mathrm P}(x)
\]
for $j=1,2,3$,
and we can repeat the previous argument.
Moreover, since $\np{\vvels,\vpress}$ satisfy the steady Oseen equations
and $f$ has compact support,
elliptic regularity implies smoothness of $\np{\vvels,\vpress}$ far from the boundary.
Therefore, the above argument is admissible for derivatives of
$\mathcal R_0$ and $\mathfrak R_0$ of arbitrary order.
\end{proof}

\begin{remark}
By assuming more time regularity of the data $f$ and $\vvel_b$,
we could ensure higher-order regularity of the full solution $(\vvel,\vpres)$ 
far from the boundary, that is, not only of the steady part.
Then the previous argument also yields decay rates 
for higher-order derivatives of the corresponding remainder terms.
\end{remark}

In the next theorem we treat the homogeneous boundary value problem containing solely a forcing of the form $\nabla \cdot {\mathds F}$. Note that specific decay rates are assumed for ${\mathds F}$ and $\nabla \cdot {\mathds F}$.

\begin{theorem} \label{thm:asymp.linF}
In system \eqref{periodicproblemlin} assume $v_b\equiv 0$ and $f \equiv 0$. 
If  ${\mathds F}\in\LS{\infty}(\torus;\WS{1}{\infty}(\Omega)^{3\times3})$ with
\begin{equation}
|{\mathds F}(t,x)| \leq C \bp{\snorm{x}\np{1+\wakefct(x)}}^{-2}, \qquad |{\mathcal P}_\perp{\mathds F}(t,x)| \leq C |x|^{-A}
\label{propF0}
\end{equation}
for some $A>0$,
then the asymptotic expansion at infinity of the velocity field 
$\vvel$ is of the form
\begin{equation}
\begin{split}
\vvel(t,x)
&= \fsolvel (\cdot,x) \ast_{\mathbb T} \left[  \int_{\Sigma}\left( \TT(\vvel,\vpres)n + n^\top {\mathds F} \right)\dd S \right] (t) 
+ {\mathcal R}(t,x),
\end{split}
\label{eq:asexp.vF}
\end{equation}
and there is $C>0$ such that
\begin{equation}\label{eq:remainder.v.linF}
\begin{aligned}
\snorm{{\mathcal R}_{0}(x)}
\leq C\bb{\snorm{x}\,\np{1+\wakefct(x)}}^{-3/2}\log_+\snorm{x},
\qquad
\snorm{{\mathcal R}_{\perp}(t,x)}
\leq C\snorm{x}^{-\min\set{A,4}}
\end{aligned}
\end{equation} 
for all $(t,x)\in\torus\times\Omega$
with $\snorm{x}$ sufficiently large. 
In particular, asymptotic expansions for the steady-state part $\vvels$ 
and the purely periodic part $\vvelp$ are given by
\begin{align}
\vvels(x) 
&= \fsolvelss (x)  \int_{\Sigma} \left[ {\mathcal P}\TT(\vvel,\vpres)(t,y)n (y) + {\mathcal P} n(y)^\top {\mathds F}(t,y) \right]\dd S(y)  + {\mathcal R}_{0}(x),
\label{eq:asexp.vsF}
\\
\begin{split}
\vvelp(t,x)
&= \left[ \fsolvelpp (\cdot,x) \ast_{\mathbb T}   \int_{\Sigma}\left[ {\mathcal P}_\perp \TT(\vvel,\vpres)(t,y)n (y) + {\mathcal P}_\perp n(y)^\top {\mathds F}(t,y)  \right]\dd S(y) \right](t) + {\mathcal R}_{\perp}(t,x).
\end{split}
\label{eq:asexp.vpF}
\end{align}
Additionally, if
\begin{equation}
 | (\nabla \cdot {\mathds F})(t,x)| \leq C \bp{\snorm{x}\np{1+\wakefct(x)}}^{-5/2},
 \label{propFper}
\end{equation}
then the asymptotic expansion at infinity of the pressure is of the form
\begin{equation}
p(t,x) = p_\infty(t) + \left(   \int_{\Sigma}\left[ \TT(\vvel,\vpres)(t,y)n (y) + n(y)^\top {\mathds F}(t,y)  \right]\dd S(y)  \right) \cdot {\mathrm P}(x)+ {\mathfrak R}(t,x),
\label{eq:asexp.pF}
\end{equation}
for some $p_\infty\in\LS{1}(\torus)$
and
\begin{equation}\label{eq:remainder.pF}
\snorm{\mathfrak R(t,x)}
\leq C\snorm{x}^{-2}
\min\setl{1,\np{1+\wakefct(x)}^{-2}\log_+\snorm{x}
+\snorm{x}^{-1}\log_+\snorm{x}
+\np{1+\wakefct(x)}^{-1}}
\end{equation}
for all $(t,x)\in\torus\times\Omega$ 
with $\snorm{x}$ sufficiently large.
\label{TheorMain2}
\end{theorem}
\begin{proof} 
We use the representation formulas from Theorem~\ref{thm:representationF} and split $\vvel$ into several parts. 
Then
\[
u(t,x):= \fsolvel \ast_{G} \left[ (\TT(v,p)n \delta_{\Sigma} \right] (t,x) = \left[  \fsolvel (\cdot,x) \ast_{\mathbb T}  \int_{\Sigma} \TT(\vvel,\vpres) n\,\dd S \right]  (t) + {\mathcal R}^\uvel (t,x)
\] 
is treated as in the previous lemma. 
For the part $\wvel$, which involves $\mathds F$, it is convenient to decompose $\wvel=\wvel^{\Omega}+\wvel^\Sigma$ with
\begin{align}
\begin{split}
\wvel^\Omega(t,x)
&:=\int_{\mathbb T}  \int_{\Omega}   \left[ \nabla \left( \fsolvel \mathsf{e}_i \right)(t-s,x  - y) \right]:  {\mathds F} (s,y)  \, \dd y \dd s \, \mathsf{e}_i
\\
&\phantom{:}=
\int_{\Omega} \left[ \nabla \left(\fsolvelss \mathsf{e}_i \right) (x  - y)  \right] :  \int_{\mathbb T}  {\mathds F} (s,y) \dd s  \, \dd y \, \mathsf{e}_i
\\
& \qquad
+  \int_{\mathbb T}  \int_{\Omega}  \left[ \nabla (\fsolvelpp \mathsf{e}_i)  (t-s,x-y) \right]  :  {\mathds F} (s,y)  \, \dd y \dd s \, \mathsf{e}_i
\\
&\phantom{:}=:
\wvel^\Omega_0(x) + \wvel^\Omega_\perp(t,x),
\end{split}
\label{eq:womega}
\\
\begin{split}
\wvel^\Sigma(t,x)
&:=\int_{\mathbb T}  \int_{\Sigma}  \fsolvel(x-y,t-s) n(y)^\top {\mathds F}(s,y)  \, \dd S(y) \dd s
\\
&\phantom{:}=
\int_{\Sigma} \fsolvelss(x-y) \int_{\mathbb T}  n(y)^\top {\mathds F}(s,y)  \dd s \, \dd S(y)
\\
&\qquad
+ \int_{\mathbb T}  \int_{\Sigma}  \fsolvelpp(t-s,x-y) n(y)^\top {\mathds F}(s,y)  \, \dd S(y) \dd s
\\
&\hphantom{:}=: \wvel^\Sigma_0(x) + \wvel^\Sigma_\perp(t,x).
\end{split}
\label{eq:wsigma}
\end{align}
Combining the decay properties \eqref{propF0} with Lemma~\ref{lem:conv.fsolvelss} and Lemma~\ref{lem:conv.fsolvelpp}, we conclude
\[
\begin{aligned}
\snorm{\wvels^\Omega(x)} 
&\leq C \bp{\snorm{x}\np{1+\wakefct\np{x}}}^{-3/2}\log_+\snorm{x},
\\
\snorm{\wvelp^\Omega(t,x)} 
&\leq C \snorm{x}^{-\min\set{A,4}}.
\end{aligned}
\]
If we define
\[
\mathcal R^{\wvel^\Sigma}(t,x)
\coloneqq\wvel^\Sigma(t,x) - 
\left[  \fsolvel (\cdot,x) \ast_{\mathbb T} \int_{\Sigma}  n^\top {\mathds F} \dd S \right] (t), 
 \]
then Lemma~\ref{lem:conv.decay} implies 
\[
\begin{aligned}
\snorml{\mathcal R^{\wvel^\Sigma}_0(x)}
&\leq C\bp{\snorm{x}\np{1+\wakefct\np{x}}}^{-3/2},
\\
\snorml{\mathcal R^{\wvel^\Sigma}_\perp(t,x)}
&\leq C\snorm{x}^{-4}
\end{aligned}
\]
for all $(t,x)\in\torus\times\Omega$ 
with $\snorm{x}$ sufficiently large. 
Since we have ${\mathcal R} = {\mathcal R}^\uvel - \wvel^\Omega - {\mathcal R}^{\wvel^\Sigma}$,
estimate \eqref{eq:remainder.v.linF} follows.

Now let us address the pressure term. 
We write \eqref{eq:repr.p.linF} in the form  $\vpres=\upres-\wpres$ with
\[
\wpres(t,x)\coloneqq
\int_{\Omega} {\mathrm P}(x-y) \cdot  (\nabla \cdot {\mathds F})(t,y) \,\dd y.
\]
As in the proof of Theorem~\ref{thm:asymp.lin}
we obtain
\begin{equation}
\snorml{\mathfrak R^\upres(t,x)}\leq C \snorm{x}^{-3}
\label{eq:est.rem.upres}
\end{equation}
for all $(t,x)\in\torus\times\Omega$,
where we split $\mathfrak R=\mathfrak R^\upres-\mathfrak R^\wpres$
with
\[
\mathfrak R^\wpres(t,x)
=\wpres(t,x)
- {\mathrm P}(x)\cdot\int_{\Sigma} n(y)^\top {\mathds F}(t,y)  \,\dd S(y).
\]
Fix $R_\ast>0$ such that $\Sigma\subset B_{R_\ast}$ and let $\chi\in\CSci(\R^3)$  be a smooth function
such that $\chi(x)=1$ for $\snorm{x}\leq R_\ast/2$
and $\chi(x)=0$ for $\snorm{x}\geq R_\ast$. Then 
\[
\begin{aligned}
\wpres(t,x)
&=\int_{\Omega} \chi(x-y) {\mathrm P}(x-y) \cdot  [ \nabla \cdot {\mathds F}(t,y) ] \,\dd y
\\
&\qquad
+\int_{\Omega} \nb{1-\chi(x-y)} \nabla {\mathrm P}(x-y) :  {\mathds F}(t,y) \,\dd y
\\
&\qquad
-\int_{\Omega} \nabla\chi(x-y)  \otimes {\mathrm P}(x-y) : {\mathds F}(t,y) \,\dd y
\\
&\qquad
+\int_{\Sigma} \chi(x-y) {\mathrm P}(x-y) \cdot \left[ n(y)^\top {\mathds F}(t,y) \right] \,\dd y
\\
&\eqqcolon \wpres_1(t,x)+\wpres_2(t,x)+\wpres_3(t,x)+\wpres_4(t,x).
\end{aligned}
\]
For the first term, we use $\fsolpres(x)=C \snorm{x}^{-2}$
and directly estimate
\[
\snorml{\wpres_1(t,x)}
\leq
\int_{B_{R^\ast}(x)} C\snorm{x-y}^{-2} \bp{\snorm{y}\np{1+\wakefct(y)}}^{-5/2} \,\dd y
\leq C \bp{\snorm{x}\np{1+\wakefct(x)}}^{-5/2}
\]
for $\snorm{x}\geq 2R_\ast$.
For the second term, Lemma~\ref{lem:conv.fsolpres}
implies
\[
\begin{aligned}
\snorml{\wpres_2(t,x)}
\leq 
C\snorm{x}^{-2}
\min\setl{1,\np{1+\wakefct(x)}^{-2}\log_+\snorm{x}
+\snorm{x}^{-1}\log_+\snorm{x}
+\np{1+\wakefct(x)}^{-1}}.
\end{aligned}
\]
For the third term, we consider $\snorm{x}\geq 2R_\ast$ and deduce
\[
\snorml{\wpres_3(t,x)}
\leq
\int_{B_{R_\ast}(x)\setminus B_{R_\ast/2}(x)} 
\snorm{x-y}^{-2} \bp{\snorm{y}\np{1+\wakefct(y)}}^{-2} \,\dd y
\leq \bp{\snorm{x}\np{1+\wakefct(x)}}^{-2}.
\]
Invoking Lemma~\ref{lem:conv.decay} for the last term, 
we conclude
\[
\snormL{\wpres_4(t,x)
-{\mathrm P}(x)\cdot\int_{\Sigma} {\mathds F}(t,y) n(y)\,\dd S(y)} \leq C\snorm{x}^{-3}
\]
for $\snorm{x}$ large.
In virtue of $\mathfrak R=\mathfrak R^\upres-\mathfrak R^\wpres$
and~\eqref{eq:est.rem.upres},
we finally arrive at~\eqref{eq:remainder.pF}.
\end{proof}

\section{Proofs of the main results}
\label{sec:proofs}

The representation formulas for the nonlinear case 
given in Theorem~\ref{thm:repr.nonlin} 
can be deduced from the linear case as follows.

\begin{proof}[Proof of Theorem~\ref{thm:repr.nonlin}]
The weak solution $\vvel$
to the nonlinear problem~\eqref{periodicproblem}
solves the linear problem~\eqref{periodicproblemlin}
if we replace ${\mathds F}$ with $ \vvel \otimes \vvel$
since we then have $f-\nabla\cdot \mathds F = f-\vvel\cdot\nabla\vvel$.
By Theorem~\ref{thm:regularity}
we deduce $\vvel\cdot\nabla\vvel\in\LS{q}(\torus\times\Omega)^3$ 
for any $q\in(1,\infty)$.
Therefore, the representation formulas~\eqref{eq:repr.v},
\eqref{eq:repr.v2} and~\eqref{eq:repr.p}
directly follow from combining those for the linear case given in Theorem~\ref{thm:representation}
and Theorem~\ref{thm:representationF}.
\end{proof}

To prove Theorem~\ref{thm:asymp.nonlin},
the linear terms in the representation formulas can be handled
as in the proof of Theorem~\ref{thm:asymp.lin}.
However, new difficulties arise from the nonlinear terms.
To treat them, 
we begin with the following decay property of the velocity gradient.
For the whole section,
we choose a radius $R_\ast>0$ such that $\Sigma\subset B_{R_\ast}$.

\begin{lemma}\label{lem:gradvel.decay}
For all $\varepsilon>0$ there exists $C>0$ such that
for all $R>R_*$ it holds
\begin{equation}\label{eq:gradvel.decay}
\int_\torus\int_{\Omega\setminus B_R}|\nabla\vvel|^2\,\dd x\dd t 
\leq C R^{-1+\varepsilon}.
\end{equation}
\end{lemma}

\begin{proof}
For the case $\Omega=\R^3$, 
this estimate was shown in~\cite[Lemma 5.2]{GaldiKyed18}.
In the present case of an exterior domain $\Omega$, 
it follows in exactly the same way
by using the increased regularity stated in Theorem~\ref{thm:regularity}.
\end{proof}

From \eqref{eq:gradvel.decay}
we can derive a first pointwise estimate 
of the remainder terms for the velocity field.
This estimate serves as a starting point 
for an iterative procedure that increases the decay rate step by step.

\begin{proof}[Proof of Theorem~\ref{thm:asymp.nonlin}]

We begin with the treatment of the velocity field
and follow the approach from~\cite{GaldiKyed18}
and~\cite{Eiter21}.
Instead of repeating these calculations here in detail,
we explain the major steps and the differences that are due to 
the occurrence of boundary terms.

To separate the linear from the nonlinear contributions in the representation formula~\eqref{eq:repr.v},
we let $\vvel = \uvel-\wvel$ with
\[
\wvel(t,x)
\coloneqq \int_{\torus\times\Omega}
\fsolvel(t-s,x-y) \nb{\vvel(s,y)\cdot\nabla\vvel(s,y)}\,\dd s \dd y.
\]
Then $\uvel$ satisfies
\[
\begin{aligned}
\uvel(t,x)  & =  {\widetilde u}_0^{(1)}(x)  +  {\widetilde u}_0^{(2)}(x)  +  {\widetilde u}_0^{(3)}(x) 
\medskip \\ 
&
\quad +  {\widetilde u}_\perp^{(1)}(t,x)  +  {\widetilde u}_\perp^{(2)}(t,x)  +  {\widetilde u}_\perp^{(3)}(t,x) 
\medskip  \\
& 
\quad  - \int_{\Sigma}  v_b(t,y) \cdot n(y) {\mathrm P}(x-y) \,\dd S(y)
\end{aligned}
\]
for $(t,x) \in {\mathbb T} \times \Omega$,
where the terms on the right-hand side are
exactly the same as in the representation formula~\eqref{eq:repr.v.lin} for the linear case.
Decomposing $\uvel=\uvels+\uvelp$ into steady-state and purely periodic part,
and arguing as in the proof of Theorem~\ref{thm:asymp.lin},
we obtain that the associated remainder term
\[
\begin{aligned}
\mathcal R^\uvel(t,x)
&\coloneqq \uvel(t,x) - 
\bb{\fsolvel(\cdot,x) \ast_{\mathbb T} \mathcal F^{\mathrm{lin}}}(t) 
+ {\mathrm P}(x) \Phi(t)
-\nabla {\mathrm P}(x) \Psi(t),
\end{aligned}
\]
where 
\[
\mathcal F^{\mathrm{lin}}(t)
\coloneqq
\int_{\Omega} f(t,y)\,\dd y 
+ \int_{\Sigma}\left[ \TT(v,p)(t,y)n (y) + \bp{\zeta \cdot n(y)}\vvel_b(t,y)\right]\,\dd S(y),
\]
satisfies
the same estimates as in the linear case,
that is,
\begin{equation}
\snorm{D_x^\alpha\mathcal R^\uvel_{0}(x)}
\leq C\bb{\snorm{x}\,\np{1+\wakefct(x)}}^{-3/2-\snorm{\alpha}/2},
\qquad
\snorm{D_x^\alpha\mathcal R^\uvel_{\perp}(t,x)}
\leq C\snorm{x}^{-4-\snorm{\alpha}}
\label{eq:est.remainder.u}
\end{equation}
for $\snorm{\alpha}\leq 2$ and $(t,x)\in\torus\times\Omega$
with $\snorm{x}$ large.
In particular, 
in virtue of the decay properties of the fundamental solutions,
this implies
\begin{align}
\snorm{D_x^\alpha\uvels(x)}
&\leq C\bb{\snorm{x}\,\np{1+\wakefct(x)}}^{-1-\snorm{\alpha}/2},
\label{eq:est.u.s}
\\
\snorm{D_x^\alpha\uvelp(t,x)}
&\leq C\snorm{x}^{-2-\snorm{\alpha}}.
\label{eq:est.u.p}
\end{align}

Next we derive estimates of $\wvel$,
which will lead to an estimate of the actual remainder term 
$\mathcal R=\mathcal R^\uvel -\mathcal R^\wvel$,
where
\[
\mathcal R^\wvel(t,x)
\coloneqq\wvel(t,x)-\int_\torus\fsolvel(t-s,x) \int_{\Sigma}\np{\vvel_b(s,y)\cdot n(y)}\vvel_b(s,y) \,\dd S(y)\dd s.
\]
To do so, we decompose $\wvel=\wvels+\wvelp$
such that
\begin{align}
\wvels(x)
&=\fsolvelss \ast_{{\mathbb R}^3}\bb{\chi_\Omega\bp{\vvels\cdot\nabla\vvels+\proj\np{\vvelp\cdot\nabla\vvelp}}}(x),
\label{eq:repr.w.s}
\\
\begin{split}
\wvelp(t,x)
&=\fsolvelpp \ast_G \bb{\chi_\Omega\bp{\vvels\cdot\nabla\vvelp + \vvelp\cdot\nabla\vvels+\projcompl\np{\vvelp\cdot\nabla\vvelp}}}(t,x).
\end{split}
\label{eq:repr.w.p}
\end{align}
Exploiting the decay properties of $\fsolvel$ and Lemma~\ref{lem:gradvel.decay},
one can argue as in~\cite{GaldiKyed18}
to first show that
\begin{equation}
\forall\varepsilon>0 \quad \exists C>0 : \ \
\snorm{\wvels(t,x)}\leq C
\snorm{x}^{-1+\varepsilon}
\label{eq:est.w.s1}
\end{equation}
and
\begin{multline}
\forall\varepsilon>0 \, \quad \exists C>0 : \quad 
\snorm{\wvelp(t,x)}\leq \medskip \\ C\bp{
\snorm{x}^{-3/2}
+\snorm{x}^{-2/5+\varepsilon}\np{
\norm{\vvels}_{\LS{\infty}(B^{\snorm{x}/2})}
+\norm{\vvelp}_{\LS{\infty}(\torus\times B^{\snorm{x}/2})}}}.
\label{eq:est.w.p1}
\end{multline}
Combining \eqref{eq:est.w.s1} with \eqref{eq:est.u.s}, we conclude
\begin{equation}
\forall\varepsilon>0 \quad \exists C>0 : \ \
\snorm{\vvels(t,x)}\leq C
\snorm{x}^{-1+\varepsilon}.
\label{eq:est.v.s1}
\end{equation}
Similarly, 
\eqref{eq:est.u.p} leads to
\begin{equation}
\forall\varepsilon>0 \quad \exists C>0 : \ \
\snorm{\vvelp(t,x)}\leq C\bp{\snorm{\wvelp(t,x)}
+\snorm{x}^{-2}
}.
\label{eq:est.v.p1}
\end{equation}
We use these inequalities to further estimate \eqref{eq:est.w.p1} and to obtain
\begin{equation}
\forall\varepsilon>0 \quad \exists C>0 : \ \ 
\snorm{\wvelp(t,x)}\leq C\bp{
\snorm{x}^{-2/5+\varepsilon}\norm{\wvelp}_{\LS{\infty}(\torus\times B^{\snorm{x}/2})}
+\snorm{x}^{-7/5+\varepsilon}}.
\label{eq:est.w.p2}
\end{equation}
Since $\vvelp\in\LS{\infty}(\torus\times\Omega)$ by Theorem~\ref{thm:regularity},
we have $\wvelp\in\LS{\infty}(\torus\times\Omega)$ due to \eqref{eq:est.u.p}.
This allows to derive a first estimate of $\wvelp$ from~\eqref{eq:est.w.p2},
We can use this estimate to derive a new bound of the right-hand side of~\eqref{eq:est.w.p2},
which improves the estimate of $\wvelp$.
Iterating this procedure and combining it with~\eqref{eq:est.v.p1}, we arrive at
\begin{equation}
\forall\varepsilon>0 \quad \exists C>0 : \ \ 
\snorm{\vvelp(t,x)} + \snorm{\wvelp(t,x)}\leq C\snorm{x}^{-7/5+\varepsilon}.
\label{eq:est.vw.p1}
\end{equation}

To improve these decay rates, 
we turn to the representation \eqref{eq:repr.v2}, 
from which we deduce that $\wvel=\wvel^\Omega+\wvel^\Sigma$
with $\wvel^\Omega$ and $\wvel^\Sigma$ 
defined as in~\eqref{eq:womega} and~\eqref{eq:wsigma}
for $\mathds F\coloneqq\vvel\otimes\vvel$, respectively.
First of all, Lemma~\ref{lem:conv.decay} implies the estimates
\begin{align}
\snorm{\wvels^\Sigma(x)}
&\leq C\bp{\snorm{x}\np{1+\wakefct\np{x}}}^{-1},
\label{eq:est.wsigma.s}
\\
\snorm{\wvelp^\Sigma(t,x)}
&\leq C\snorm{x}^{-3}
\label{eq:est.wsigma.p}
\end{align}
for $\snorm{x}$ large.
To obtain a decay rate for $\wvelp^\Omega$,
we use~\eqref{eq:est.v.s1} and~\eqref{eq:est.vw.p1},
which imply
\[
\snorm{\vvels\otimes\vvelp+\vvelp\otimes\vvels+\projcompl\np{\vvelp\otimes\vvelp}}(t,x)
\leq C\bp{\snorm{x}^{-12/5+\varepsilon}+\snorm{x}^{-14/5+\varepsilon}}.
\]
Now Lemma~\ref{lem:conv.fsolvelpp} yields 
the estimate
\begin{equation}
\forall\varepsilon>0 \quad  \exists C>0 : \ \ 
\snorm{\wvelp^\Omega(t,x)}\leq C\snorm{x}^{-12/5+\varepsilon}.
\label{eq:est.w.p3}
\end{equation}
In virtue of~\eqref{eq:est.v.p1} and~\eqref{eq:est.wsigma.p}, this yields the (optimal) estimate
\begin{equation}
\snorm{\vvelp(t,x)}
\leq C\snorm{x}^{-2}
\label{eq:est.v.p3}
\end{equation}
for all $(t,x)\in\torus\times\Omega$
with $\snorm{x}$ large. 
For $\wvels^\Omega$ we proceed similarly.
From~\eqref{eq:est.v.s1} and \eqref{eq:est.v.p3}
we deduce
\[
\snorm{\vvels\otimes\vvels+\proj\np{\vvelp\otimes\vvelp}} (x)
\leq C\bp{\snorm{x}^{-2+\varepsilon}+\snorm{x}^{-4}},
\]
so that Lemma~\ref{lem:conv.fsolvelss} yields
\begin{equation}
\forall\varepsilon>0 \quad \exists C>0 : \ \ 
\snorm{\wvels^\Omega(x)}\leq C
\snorm{x}^{-1+\varepsilon}\np{1+\wakefct\np{x}}^{-1/2+\varepsilon}.
\label{eq:est.w.s2}
\end{equation}
Combining this estimate with~\eqref{eq:est.u.s} and~\eqref{eq:est.wsigma.s},
we conclude
\begin{equation}
\forall\varepsilon>0 \quad \exists C>0 : \ \ 
\snorm{\vvels(x)}
\leq C\snorm{x}^{-1+\varepsilon}\np{1+\wakefct\np{x}}^{-1/2+\varepsilon},
\label{eq:est.v.s2}
\end{equation}
which leads to the improved pointwise estimate
\[
\snorm{\vvels\otimes\vvels+\proj\np{\vvelp\otimes\vvelp}} (x)
\leq C\bp{\snorm{x}^{-2+\varepsilon}\np{1+\wakefct\np{x}}^{-1+\varepsilon}+\snorm{x}^{-4}}.
\]
Repeating the above argument and using Lemma~\ref{lem:conv.fsolvelss} again,
we deduce
\begin{equation}
\forall\varepsilon>0 \quad \exists C>0 : \ \ 
\snorm{\wvels(x)}\leq C
\snorm{x}^{-3/2+\varepsilon}\np{1+\wakefct\np{x}}^{-1+\varepsilon},
\label{eq:est.w.s3}
\end{equation}
and with~\eqref{eq:est.u.s} and~\eqref{eq:est.wsigma.s}
we conclude the (optimal) decay rate 
\begin{equation}
\snorm{\vvels(x)}
\leq C\bp{\snorm{x}\np{1+\wakefct\np{x}}}^{-1}
\label{eq:est.v.s3}
\end{equation}
for all $x\in\Omega$.
Invoking now~\eqref{eq:est.v.p3} and~\eqref{eq:est.v.s3},
we obtain
\[
\snorm{\vvels\otimes\vvels+\proj\np{\vvelp\otimes\vvelp}} (x)
\leq C\bp{\bp{\snorm{x}\np{1+\wakefct\np{x}}}^{-2}+\snorm{x}^{-4}}
\leq C\bp{\snorm{x}\np{1+\wakefct\np{x}}}^{-2}
\]
and 
\[
\snorm{\vvels\otimes\vvelp+\vvelp\otimes\vvels+\projcompl\np{\vvelp\otimes\vvelp}}(t,x)
\leq C\bp{\snorm{x}^{-3}\np{1+\wakefct\np{x}}^{-1}+\snorm{x}^{-4}}
\leq C\snorm{x}^{-3}.
\]
Since $\vvel$ is a solution to the linear problem~\eqref{periodicproblemlin}
together with $\mathds F=\vvel\otimes\vvel$,
the asymptotic expansion~\eqref{eq:asexp.vel}
with
the asserted remainder estimates~\eqref{eq:rem.vs} and~\eqref{eq:rem.vp}
now follows from Theorem~\ref{thm:asymp.lin} and Theorem~\ref{thm:asymp.linF}.

Now we turn to the asymptotic expansion of $\nabla \vvel$
and derive estimates of $\nabla\wvel$.
From the formulas~\eqref{eq:repr.w.s}
and~\eqref{eq:repr.w.p}
we deduce
\begin{align}
\partial_j\wvels(x)
&=\partial_j\fsolvelss\ast\bb{\chi_\Omega\bp{\vvels\cdot\nabla\vvels+\proj\np{\vvelp\cdot\nabla\vvelp}}}(x),
\label{eq:repr.dw.s}
\\
\partial_j\wvelp(t,x)
&=\partial_j\fsolvelpp\ast\bb{\chi_\Omega\bp{\vvels\cdot\nabla\vvelp + \vvelp\cdot\nabla\vvels+\projcompl\np{\vvelp\cdot\nabla\vvelp}}}(t,x)
\label{eq:repr.dw.p}
\end{align}
for $j=1,2,3$. 
We use the integrability properties 
from Theorem~\ref{thm:regularity} 
and the estimates~\eqref{eq:est.v.s2} and~\eqref{eq:est.v.p3}
to proceed analogously to~\cite{Eiter21}
and to conclude the first estimate
\[
\snorm{\nabla\wvels(x)}
+\snorm{\nabla\wvelp(t,x)}
\leq C\snorm{x}^{-1},
\]
and thus
\[
\snorm{\nabla\vvels(x)}
+\snorm{\nabla\vvelp(t,x)}
\leq C\snorm{x}^{-1}
\]
in virtue of~\eqref{eq:est.u.s} and~\eqref{eq:est.u.p}.
Together with~\eqref{eq:est.v.s2} and~\eqref{eq:est.v.p3}, this estimate implies
\[
\begin{aligned}
\snorml{\vvels\cdot\nabla\vvels+\proj\np{\vvelp\cdot\nabla\vvelp}}(x)
&\leq C\bp{\snorm{x}^{-2}\np{1+\wakefct(x)}^{-1}+\snorm{x}^{-3}}
\\
&\leq C\snorm{x}^{-2}\np{1+\wakefct(x)}^{-1/2},
\\
\snorml{\vvels\cdot\nabla\vvelp+\vvelp\cdot\nabla\vvels+\projcompl\np{\vvelp\cdot\nabla\vvelp}}(t,x)
&\leq C\bp{\snorm{x}^{-2}\np{1+\wakefct(x)}^{-1}
+\snorm{x}^{-3}+\snorm{x}^{-3}}
\\
&\leq C\snorm{x}^{-2}.
\end{aligned}
\]
Returning now to~\eqref{eq:repr.dw.s} and~\eqref{eq:repr.dw.p},
we can use Lemma~\ref{lem:conv.fsolvelss} and Lemma~\ref{lem:conv.fsolvelpp}
to deduce
\[
\begin{aligned}
\snorm{\nabla\wvels(x)}
&\leq C\snorm{x}^{-5/4}\np{1+\wakefct(x)}^{-3/4},
\\
\snorm{\nabla\wvelp(t,x)}
&\leq C\snorm{x}^{-2},
\end{aligned}
\]
and thus
\[
\begin{aligned}
\snorm{\nabla\vvels(x)}
&\leq C\snorm{x}^{-5/4}\np{1+\wakefct(x)}^{-3/4},
\\
\snorm{\nabla\vvelp(t,x)}
&\leq C\snorm{x}^{-2},
\end{aligned}
\]
due to~\eqref{eq:est.u.s} and~\eqref{eq:est.u.p}.
Again, this gives an improved estimate on the nonlinear terms,
and repeating this argument once more,
we finally arrive at the (optimal) estimate
\begin{equation}
\begin{aligned}
\snorm{\nabla\vvels(x)}
&\leq C\snorm{x}^{-3/2}\np{1+\wakefct(x)}^{-3/2},
\\
\snorm{\nabla\vvelp(t,x)}
&\leq C\snorm{x}^{-3}.
\end{aligned}
\label{eq:est.dv}
\end{equation}
To derive estimates~\eqref{eq:rem.dvs} and~\eqref{eq:rem.dvp}
for the corresponding remainder terms,
we introduce another decomposition of $\mathcal R^\wvel$.
Let $\chi\in\CSci(\R^3)$  a smooth function
such that $\chi(x)=1$ for $\snorm{x}\leq R_\ast/2$
and $\chi(x)=0$ for $\snorm{x}\geq R_\ast$.
We decompose $\partial_j\mathcal R^\wvel=I+J$, $j=1,2,3$, with
\[
\begin{aligned}
I(t,x)
&\coloneqq\int_{\torus\times\Omega}
\chi(x-y)\,\partial_j\fsolvel(t-s,x-y) \nb{\vvel(s,y)\cdot\nabla\vvel(s,y)}\,\dd(s,y),
\\
J(t,x)
&\coloneqq\int_{\torus\times\Omega}
\nb{1-\chi(x-y)}\,\partial_j\fsolvel(t-s,x-y) \nb{\vvel(s,y)\cdot\nabla\vvel(s,y)}\,\dd(s,y)
\\
&\qquad\qquad\qquad\qquad
-\int_\torus\partial_j\fsolvel(t-s,x) \int_{\Sigma}\np{\vvel_b(s,y)\cdot n(y)}\vvel_b(s,y) \,\dd S(y)\dd s.
\end{aligned}
\]
By~\eqref{eq:est.v.s3},~\eqref{eq:est.v.p3} and~\eqref{eq:est.dv},
we have
\[
\begin{aligned}
\snorml{\vvels\cdot\nabla\vvels+\proj\np{\vvelp\cdot\nabla\vvelp}}(x)
&\leq C\snorm{x}^{-5/2}\np{1+\wakefct(x)}^{-5/2},
\\
\snorml{\vvels\cdot\nabla\vvelp+\vvelp\cdot\nabla\vvels+\projcompl\np{\vvelp\cdot\nabla\vvelp}}(t,x)
&\leq C\snorm{x}^{-7/2}\np{1+\wakefct(x)}^{-3/2}.
\end{aligned}
\]
Since $\partial_j\fsolvel\in\LSloc{1}(\torus\times\R^3)$ 
(see~\cite{G,EiterKyed18})
and $\chi\partial_j\fsolvel$ has compact support, 
Lemma~\ref{lem:conv.decay}
implies
\[
\begin{aligned}
\snorm{I_0(x)}
&\leq C\snorm{x}^{-5/2}\np{1+\wakefct(x)}^{-5/2},
\\
\snorm{I_\perp(t,x)}
&\leq C\snorm{x}^{-7/2}\np{1+\wakefct(x)}^{-3/2}
\end{aligned}
\]
for $\snorm{x}$ sufficiently large.
For $J_0$ we use 
$\vvel\cdot\nabla\vvel=\nabla \cdot (\vvel\otimes\vvel)$ and integration by parts
to obtain the decomposition $J_0=J_0^1+J_0^2+J_0^3$ with
\[
\begin{aligned}
\snorm{J_0^1(x)}
&\leq \int_{\Omega}
\nb{1-\chi(x-y)}\,\snorml{\partial_j \nabla\fsolvelss(x-y)} \,
\snorml{\vvels(y)\otimes\vvels(y)+\proj\np{\vvelp\otimes\vvelp}(y)}\,\dd y,
\\
\snorm{J_0^2(x)}
&\leq
\int_{\Omega}
\snorm{\nabla\chi(x-y)}\,\snorml{\partial_j\fsolvelss(x-y)} \,
\snorml{\vvels(y)\otimes\vvels(y)+\proj\np{\vvelp\otimes\vvelp}(y)}\,\dd y,
\\
\snorm{J_0^3(x)}
&=
\snormL{\int_{\Sigma}
\nb{1-\chi(x-y)}\,\partial_j\fsolvelss(x-y)\, 
\proj\np{\np{\vvel_b\cdot n}\vvel_b}(y)\,\dd y
\\
&\qquad\qquad\qquad\qquad\qquad\qquad
-\partial_j\fsolvelss(x) \int_{\Sigma}\proj\bp{\np{\vvel_b\cdot n}\vvel_b}(y) \,\dd S(y)}.
\end{aligned}
\]
Now consider $\snorm{x}\geq 2R_\ast$. 
For $\snorm{y}\leq R_\ast$,
we then have $\snorm{x-y}\geq\snorm{x}-\snorm{y}\geq \snorm{x}/2\geq R_\ast$.
Using Lemma~\ref{lem:conv.decay}, we deduce
\[
\snorml{J_0^3(x)}
\leq \bp{\snorm{x}\np{1+\wakefct\np{x}}}^{-2}
\]
for $\snorm{x}$ sufficiently large.
Similarly, for $J_0^2$ we deduce
\[
\begin{aligned}
\snorm{J_0^2(x)}
&\leq C\int_{B_{R_\ast}(x)\setminus B_{R_\ast/2}(x)}
\snorml{\partial_j\fsolvelss(x-y)} \,
\bp{\snorm{y}\np{1+\wakefct\np{y}}}^{-2}\,\dd y
\leq C\bp{\snorm{x}\np{1+\wakefct\np{x}}}^{-2}
\end{aligned}
\]
for large $\snorm{x}$
since $\partial_j\fsolvelss(x-y)\in\LSloc{1}(\R^3)$.
To estimate $J_0^1$, we use~\eqref{eq:decay.fsolvelss} 
and the decay estimates~\eqref{eq:est.v.s3} and~\eqref{eq:est.v.p3}
to conclude
\[
\begin{aligned}
\snorm{J_0^1(x)}
&\leq C\int_{\Omega\setminus B_{R_\ast/2}(x)}
\bp{\snorm{x-y}\np{1+\wakefct\np{x-y}}}^{-2}
\bp{\snorm{y}\np{1+\wakefct\np{y}}}^{-2}\,\dd y
\\
&\leq C \bp{\snorm{x}\np{1+\wakefct\np{x}}}^{-2}
\log_+\Bp{\frac{\snorm{x}}{1+\wakefct(x)}}
\end{aligned}
\]
for $\snorm{x}\geq 2R_\ast$,
where we used~\cite[Lemma 3.5]{Eiter21}.
Collecting all estimates
for $\partial_j\mathcal R^{\wvel}_0=I_0+J_0^1+J_0^2+J_0^3$
and combining them with~\eqref{eq:est.remainder.u},
we obtain~\eqref{eq:rem.dvs}.
Finally, we estimate $J_\perp$ directly,
which yields
\[
\begin{aligned}
\snorm{J_\perp(t,x)}
&\leq C\bp{\int_{\Omega\setminus B_1(x)}\snorm{x-y}^{-4}\snorm{y}^{-7/2}\np{1+\wakefct(y)}^{-3/2}\,\dd y
+\snorm{x}^{-4}}
\\
&\leq C\snorm{x}^{-7/2}\np{1+\wakefct(x)}^{-1/2}
\end{aligned}
\]
for $\snorm{x}\geq 2R_\ast$
due to~\eqref{eq:decay.fsolvelpp}.
Together with the estimate of $I_\perp$ and~\eqref{eq:est.remainder.u},
this gives~\eqref{eq:rem.dvp}.

The results for the pressure, which can be decomposed into a component associated with $v_b$ and $f$ and another one associated with ${\mathds F} = v \otimes v$, are now obtained by combining  \eqref{eq:asexp.p}--\eqref{eq:remainder.p} and \eqref{eq:asexp.pF}--\eqref{eq:remainder.pF}.
\end{proof}

The statements of Corollary~\ref{cor:asymp.nonlin.split} and Corollary~\ref{cor:decay.general}
are now direct consequences.

\begin{proof}[Proof of Corollary~\ref{cor:asymp.nonlin.split}]
From the identity $\fsolvel=1_\torus\otimes\fsolvelss+\fsolvelpp$ 
we directly conclude~\eqref{eq:asexp.velp} from~\eqref{eq:asexp.vel}
by taking the purely periodic part.
Taking the steady-state part, 
we nearly arrive at a similar formula,
which only differs from~\eqref{eq:asexp.vels}
by the term $\Psi_0\cdot \nabla P(x)$.
Since $\snorm{\nabla P(x)}\leq C\snorm{x}^{-3}$,
this term can be absorbed into the steady-state remainder $\calR_0$.
\end{proof}

\begin{proof}[Proof of Corollary~\ref{cor:decay.general}]
The asserted estimates were already shown in the process of proving Theorem~\ref{thm:repr.nonlin}.
However, they are also direct consequences of the asymptotic expansions~\eqref{eq:asexp.vels},~\eqref{eq:asexp.velp} and~\eqref{eq:asexp.pres}
combined with the decay properties of the fundamental solutions from~\eqref{eq:decay.fsolvelss},~\eqref{eq:decay.fsolvelpp} 
and~\eqref{eq:decay.fsolpres}, respectively.
\end{proof}

Finally, we consider the case of constant total flux. 

\begin{proof}[Proof of Theorem~\ref{thm:constflux}]
First of all, note that $\partial_t\Phi=0$ is equivalent to $\Phi\equiv\Phi_0$, and thus to $\Phi_\perp=0$.
Therefore, the asymptotic expansions from Theorem~\ref{thm:asymp.nonlin}
and Corollary~\ref{cor:asymp.nonlin.split} 
simplify, and we obtain~\eqref{eq:asexp.vel.constflux} and~\eqref{eq:asexp.pres.constflux}.
Moreover, these formulas imply the improved pointwise estimates~\eqref{eq:decay.constflux.vs}--\eqref{eq:decay.constflux.p}
in virtue of the decay properties of the fundamental solutions from~\eqref{eq:decay.fsolvelpp} and~\eqref{eq:decay.fsolpres}.
In particular, this yields
\[
\snorm{\vvels\otimes\vvelp+\vvelp\otimes\vvels+\projcompl\np{\vvelp\otimes\vvelp}}(t,x)
\leq C\bp{\snorm{x}^{-4}\np{1+\wakefct\np{x}}^{-1}+\snorm{x}^{-6}}
\leq C\snorm{x}^{-4},
\]
and we can conclude 
the remainder estimate~\eqref{eq:rem.vp.constflux}
for $\mathcal R_\perp$ 
from a combination of Theorem~\ref{thm:asymp.lin}
and Theorem~\ref{thm:asymp.linF}
as above.
Moreover, we now have
\[
\snorml{\vvels\cdot\nabla\vvelp+\vvelp\cdot\nabla\vvels+\projcompl\np{\vvelp\cdot\nabla\vvelp}}(t,x)
\leq C\snorm{x}^{-9/2}\np{1+\wakefct(x)}^{-3/2},
\]
so that, similarly to above, 
\[
\begin{aligned}
&\snormL{\int_{\torus\times\Omega}
\nb{1-\chi(x-y)}\,\partial_j\fsolvelpp(t-s,x-y) \projcompl\nb{\vvel(s,y)\cdot\nabla\vvel(s,y)}\,\dd(s,y)}
\\
&\qquad
\leq\int_{\Omega\setminus B_2(x)}\snorm{x-y}^{-4}\snorm{y}^{-9/2}\np{1+\wakefct(y)}^{-3/2}\,\dd y
\leq \snorm{x}^{-9/2}\np{1+\wakefct(x)}^{-3/2}
\end{aligned}
\]
for $\snorm{x}\geq 2R_\ast$.
Therefore,
the estimate of $J_\perp$
in the proof of Theorem~\ref{thm:asymp.nonlin}
can be replaced with
\[
\begin{aligned}
\snorm{J_\perp(t,x)}
\leq C\bp{\snorm{x}^{-9/2}\np{1+\wakefct(x)}^{-3/2}
+\snorm{x}^{-4}}
\leq C\snorm{x}^{-4},
\end{aligned}
\]
and in virtue of~\eqref{eq:est.remainder.u} and the previous estimate of $I_\perp$,
we obtain~\eqref{eq:rem.dvp}.

To derive the asymptotic expansion~\eqref{eq:asexp.dvp.constflux},
observe that 
$\mathcal R_\perp^{\partial_j}=\partial_j\mathcal R_\perp^\uvel-\partial_j\wvelp$.
Due to
\[
\begin{aligned}
\partial_j\wvel(t,x)
&=
I(t,x)
+\int_{\torus\times\Omega}
\nb{1-\chi(x-y)}\,\partial_j\fsolvel(t-s,x-y) \nb{\vvel(s,y)\cdot\nabla\vvel(s,y)}\,\dd(s,y),
\end{aligned}
\]
the previous integral estimate and the estimate of $I_\perp$
yield
\[
\snorml{\mathcal R_\perp^{\partial_j}(t,x)}
\leq \snorml{\partial_j\mathcal R_\perp^\uvel(t,x)}+\snorm{\partial_j\wvelp(t,x)}
\leq \snorm{x}^{-9/2}\np{1+\wakefct(x)}^{-3/2}
\]
by~\eqref{eq:est.remainder.u},
which completes the proof of Theorem~\ref{thm:constflux}.
\end{proof}

\end{document}